\documentclass[a4paper,10pt]{siamart220329}

\usepackage{graphicx}
\usepackage{color}
\usepackage{xcolor}
\usepackage{comment}
\usepackage{amsmath,mathtools}
\usepackage{amsfonts}
\usepackage{mathrsfs}
\usepackage{color}
\usepackage{algorithm}
\usepackage{algorithmic}
\usepackage{siunitx}
\usepackage{multirow}
\usepackage{stackrel}
\usepackage{booktabs}
\usepackage{tabularx}
\usepackage{arydshln}
\usepackage{cite}
\usepackage{url}
\usepackage{enumitem}

\definecolor{airforceblue}{rgb}{0.36, 0.24, 1.00}
        \definecolor{aqua}{rgb}{0.0, 1.0, 1.0}
        \definecolor{brightgreen}{rgb}{0.4, 1.0, 0.0}
        \definecolor{darkgreen}{rgb}{0.4, 0.7, 0.0}
       \definecolor{mygreen}{rgb}{0.1,.7,0.1}
       \definecolor{mygreen}{rgb}{0,0,0}
       \definecolor{red}{rgb}{0,0,0}

\newcommand{\red}{\color{red}}
\renewcommand{\r}{}

\newcommand{\tensorOne}[1]{\boldsymbol{#1}}
\newcommand{\tensorTwo}[1]{\boldsymbol{#1}}
\newcommand{\tensorFour}[1]{\mathbb{#1}}

\newtheorem{Theorem}{Theorem}

\newtheorem{Lemma}{Lemma}

\newtheorem{Remark}{Remark}

\newcommand{\R}{\mathbb{R}}
\usepackage[top=3.0cm,bottom=3.0cm,left=2.4cm,right=2.4cm]{geometry}
\usepackage{lineno}

\begin{document}
\title{Triangular preconditioners for double saddle point linear systems arising in the mixed form of poroelasticity equations}
\author{Luca Bergamaschi\footnotemark[1],
Massimiliano Ferronato\footnotemark[1] \and
\'Angeles Mart\'inez\footnotemark[2]}
\date{}
\footnotetext[1]{Department of Civil Environmental and Architectural Engineering, University of Padua, Italy.
\texttt{E-mail}: \{luca.bergamaschi, massimiliano.ferronato\}@unipd.it}
\footnotetext[2]{Department of Mathematics, Informatics and Geosciences, University of Trieste, Italy,  
\texttt{E-mail}: amartinez@units.it.
Corresponding author.}
\maketitle
\begin{abstract}
In this paper, we study a class of inexact block triangular preconditioners 
for double saddle-point symmetric linear systems arising from the mixed finite element and mixed hybrid finite element
discretization of Biot's poroelasticity equations. We develop a spectral analysis of the preconditioned matrix, showing 
that the complex eigenvalues lie in a circle of center $(1,0)$ and radius smaller than 1. In contrast, the real eigenvalues are described 
in terms of the roots of a third-degree polynomial with real coefficients. 
The results of numerical experiments are reported to show
the quality of the theoretical bounds and illustrate the efficiency of the proposed preconditioners used with GMRES, 
especially in comparison with similar block diagonal preconditioning strategies along with the MINRES iteration. 

\end{abstract}

\medskip


\smallskip


\medskip

\section{Introduction}

Given positive integer dimensions $n, m$ and $p$ with  $n \ge \max\{m,p\}$, consider the $(n+m+p) \times (n+m+p)$ double
saddle-point linear system
\begin{equation}\label{Eq1}
\mathcal{A}w\equiv\begin{bmatrix}
	A & B^{T} & 0\\
	B & -D & C^{T}\\
        0 & C & E\end{bmatrix}
	\begin{bmatrix}x\\y\\z  \end{bmatrix}=\begin{bmatrix}f\\g
\\ h\end{bmatrix}\equiv b,
\end{equation}
where $A\in \mathbb{R}^{n \times n}$ \textcolor{red}{and $E \in \mathbb{R}^{p \times p}$} are 
symmetric positive definite (SPD) matrix, $B \in \mathbb{R}^{m \times n}$ has full row 
rank, $C \in \mathbb{R}^{p \times m}$ has full rank, $D \in \mathbb{R}^{m \times m}$ is 
a square positive semi-definite (SPSD) matrix. Moreover, $f \in \mathbb{R}^{n}$, $g \in \mathbb{R}^{m}$ and $h \in \mathbb{R}^{p}$ are given vectors. Such linear systems arise in several scientific applications, including constrained least squares problems 
\cite{Yuan}, constrained quadratic programming \cite{Han}, magma-mantle dynamics \cite{Rhebergen}, to mention a few; see, 
e.g.~\cite{Chen, Cai}. Similar block structures arise in liquid crystal director modeling or the coupled Stokes-Darcy problem, and the preconditioning of such linear systems has been considered in \cite{ChenRen, Szyld, Benzi2018, BeikBenzi2022, 
Balani-et-al-2023a}. 
Symmetric positive definite block preconditioners for (\ref{Eq1}) have been
deeply studied in \cite{Bradley,Bergamaschi-COAP-2024}.

The double saddle-point structure \eqref{Eq1} also arises in the mixed form of Biot's poroelasticity equations. Biot's
model, dating back to \cite{Bio41} for its original derivation, couples Darcy's flow of a single-phase fluid with the
quasi-static mechanical deformation of an elastic, fully saturated porous medium. In its mixed form, Darcy's law is
explicitly considered as a governing equation along with the fluid mass balance, giving rise to the so-called
three-field formulation \cite{Lip02, PhiWhe07a, PhiWhe07b, JhaJua07,FerCasGam10, Hu_etal17,HonKra18}, where the solid phase 
displacements, the pore fluid pressure, and Darcy's velocity are the independent unknown quantities. Biot's poroelasticity 
model is key in several relevant applications, ranging from geosciences, such as groundwater hydrology, geothermal 
energy extraction, and geological carbon and hydrogen storage problems 
\cite{TeaFerGamGon06,LuoZen11,PanFreDouZakSheCutTer15,FerGamJanTea10,Whi_etal14,RameshKumar2023583}, to biomedicine 
\cite{Cowin99,Fri00,Kraus2023B802}. The use of a three-field mixed formulation allows for obtaining locally
mass-conservative velocity fields and proves to be particularly robust to strong jumps in the distribution
of the material parameters, especially in the permeability tensor, which makes it particularly interesting in many applications. For this reason, the three-field formulation has attracted increasing attention in recent years.

One of the most relevant issues posed by the discrete Biot model from the computational point of view is the solution
to a sequence of double saddle-point linear systems in the form \eqref{Eq1}, where the diagonal block $D$ is usually singular
or numerically close to being so, and $E$ is symmetric positive definite. Most of the applications of the three-field
formulation focused on iterative solution schemes with sequential-implicit approaches, 
e.g.~\cite{GirKumWhe16,Bot_etal17,DanWhe18,Hong2020916}, where the
poromechanical equilibrium and Darcy's flow subproblems are solved independently, iterating until convergence. Though
attractive for their simplicity and ease of implementation, sequential-implicit approaches have a two-fold limitation: (i) suitable splitting strategies, such as the so-called fixed-stress algorithm \cite{KimTchJua11a,KimTchJua11b}, 
should be employed to ensure unconditional convergence, and (ii) convergence of these methods is in any case linear. In 
recent years, an increasing interest has been shown in the development of specific classes of preconditioners to allow for the
efficient monolithic solution of the overall $3\times 3$ block double saddle-point system by Krylov subspace methods.
Spectrally equivalent block diagonal preconditioners were first advanced \cite{LeeMarWin17,TurArb14,Adl_etal20}, then a family
of block triangular and block strategies, combining physically based with fully algebraic Schur complement approximations, was
proposed \cite{CasWhiFer16,FerFraJanCasTch19,FriCasFer19,Bui_etal20,Frigo-et-al}.
A review of the block preconditioning methods proposed for the fully implicit monolithic solution of the three-field formulation of Biot's model is available in \cite{Ferronato2022}.

This work analyzes the spectral properties of the class of block triangular preconditioners developed in
\cite{Ferronato2021419,Frigo202136} for the solution of the discrete stabilized mixed and mixed hybrid finite element
formulation of Biot's poroelasticity equations. These block triangular preconditioners are based on both
algebraic and physics-based approximations of the Schur complements and generally prove quite efficient in various applications. In this paper, we derive
bounds on the eigenvalues of the corresponding preconditioned matrix, extending the results
provided in \cite{Balani-et-al-2023b}, where the simpler case with $D \equiv 0$ and $E\equiv 0$ is addressed. 
The theoretical bounds are validated on some test cases dealing with the
three-field formulation of Biot's model. The performance of the inexact block triangular preconditioner in the
GMRES acceleration ~\cite{saad1986gmres} is also compared to that of the related block diagonal variants,
which can exploit the cheaper MINRES iteration \cite{minres}. Finally, the efficiency of the proposed inexact preconditioners is tested in large-scale 3D realistic problems.

\section{Eigenvalue analysis of the inexact triangular preconditioner}
\label{sec:eig_tri}
Let us consider the prototype double saddle-point linear system \eqref{Eq1}. We define the first- and second-level
Schur complements 
\[S = D + BA^{-1}B^{T}, \;\;\; X = E + C S^{-1} C^T, \] 
and denote as $\widehat{A}, \widehat{S}$ and $\widehat{X}$ some related symmetric positive definite approximations of $A, S$, 
and $X$, respectively. In this section, we aim to analyze the eigenvalue distribution of the preconditioned matrix 
$\mathcal{A}{\mathcal{P}}^{-1}$, where
\begin{equation}\label{Eq29}
{\mathcal{P}}=\begin{bmatrix} \widehat A &B^{T}&0\\0&-\widehat S &C^{T}\\0&0& \widehat X\end{bmatrix}.
\end{equation}
The spectral properties of $\mathcal{A}{\mathcal{P}}^{-1}$ will be expressed as a function of the eigenvalues of 
$\widehat A^{-1} A, \widehat S^{-1} \widetilde{S}$ and $\widehat X^{-1} \widetilde{X}$, where 
\[ \widetilde{S} = D + B \widehat{A}^{-1} B^T, \;\;\; \widetilde{X} = E + C \widehat{S}^{-1} C^T.\]
Of course, the closer are $\widehat{A}$ and $\widehat{S}$ to $A$ and $S$, respectively, the closer are $\widetilde{S}$ and
$\widetilde{X}$ to $S$ and $X$.

Let 
\begin{equation}\label{Eq30}
{\mathcal{P}_D}=\begin{bmatrix}\widehat A&0&0\\0&\widehat S&0\\0&0&\widehat X\end{bmatrix}.
\end{equation}
Then, finding the eigenvalues of $\mathcal{A}{\mathcal{P}}^{-1}$ is equivalent to solving
$${\mathcal{P}_D}^{-\frac{1}{2}}\mathcal{A}{\mathcal{P}_D}^{-\frac{1}{2}}{w}=\lambda {\mathcal{P}_D}^{-\frac{1}{2}}{\mathcal{P}}{\mathcal{P}_D}^{-\frac{1}{2}}{w},$$
or 
\begin{equation}\label{Eq31}
	\begin{bmatrix}\overline{A}  &  R^T  &  0\\
		R  &  -\overline D  &  K^T\\
		0  &  K  & \overline E \end{bmatrix}
		\begin{bmatrix}x\\y\\z  \end{bmatrix}=\lambda \begin{bmatrix}I  &  R^T  &  0\\0  &  -I  &  K^T\\0  &  0  &  I\end{bmatrix}
\begin{bmatrix}x\\y\\z  \end{bmatrix},
\end{equation}
where $\overline{A}=\widehat{A}^{-\frac{1}{2}}A\widehat{A}^{-\frac{1}{2}},\: 
\overline{D}=\widehat{S}^{-\frac{1}{2}}D\widehat{S}^{-\frac{1}{2}},\: 
\overline{E}=\widehat{X}^{-\frac{1}{2}}E\widehat{X}^{-\frac{1}{2}},\: R =\widehat{S}^{-\frac{1}{2}}B\widehat{A}^{-\frac{1}{2}}$
and $K = \widehat{X}^{-\frac{1}{2}}C\widehat{S}^{-\frac{1}{2}}$. 
Notice that 
\begin{eqnarray} 
	R R^T &=&  \widehat S^{-\frac{1}{2}}  \left(\widetilde S -D\right) \widehat S^{-\frac{1}{2}} = \widehat S^{-\frac{1}{2}} \widetilde S \widehat S^{-\frac{1}{2}} -\overline D = \overline{S}-\overline{D}, \label{EqRR} \\
	K K^T &=&  \widehat X^{-\frac{1}{2}}  \left(\widetilde X -E\right) \widehat X^{-\frac{1}{2}} = \widehat X^{-\frac{1}{2}} \widetilde X \widehat X^{-\frac{1}{2}} -\overline E = \overline{X}-\overline{E}. \label{EqKK}
\end{eqnarray}
We define the Rayleigh quotient for a given square matrix $H$ and a nonzero vector $w$ as
\[ q(H,w) = \frac{w^* H w}{w^* w}.\]
The following indicators will be used:
\[
	\begin{array}{llll}
	\label{indA}
		0<& \gamma_{\min}^A  \equiv \lambda_{\min} (\widehat A^{-1} A),  & \gamma_{\max}^A  \equiv \lambda_{\max} (\widehat A^{-1} A), & \gamma_A(w) = q(\overline A, w) \in [ \gamma_{\min}^A,  \gamma_{\max}^A ]\equiv I_A,  \\[.6em]
		0 <& \gamma_{\min}^{S}  \equiv \lambda_{\min} (\widehat S^{-1} \widetilde{S}),  & \gamma_{\max}^{S}  \equiv \lambda_{\max} (\widehat S^{-1} \widetilde{S}), & 
	\gamma_{S}(w) = q(
		\overline{S}, w)
		\in [ \gamma_{\min}^{S},  \gamma_{\max}^{S} ] \equiv I_S,  \\[.6em]
		0 <& \gamma_{\min}^{X}  \equiv \lambda_{\min} (\widehat X^{-1} \widetilde{X}),  & \gamma_{\max}^{X} \equiv \lambda_{\max} (\widehat X^{-1} \widetilde{X}), & 
	\gamma_{X}(w) = q(
		\overline{X}, w)
		\in [ \gamma_{\min}^{X},  \gamma_{\max}^{X} ] \equiv I_X,   \\[.6em]
		0 
		\leq & \gamma_{\min}^{D}  \equiv \lambda_{\min} (\widehat S^{-1} {D}),  & \gamma_{\max}^{D}  \equiv \lambda_{\max} (\widehat S^{-1} {D}), & 
	\gamma_{D}(w) =q(\overline D ,w)
		\in [ \gamma_{\min}^{D},  \gamma_{\max}^{D} ] \equiv I_D,  \\[.6em]
		0 <& \gamma_{\min}^{E}  \equiv \lambda_{\min} (\widehat X^{-1} {E}),  & \gamma_{\max}^{E} \equiv \lambda_{\max} (\widehat X^{-1} {E}), & \gamma_{E}(w)  = 
	q(\overline E ,w)
		\in [ \gamma_{\min}^{E},  \gamma_{\max}^{E} ] \equiv I_E,   \\[.6em]
		0 <& 	\gamma_{\min}^R  \equiv \lambda_{\min} (R R^T),  & \gamma_{\max}^R  \equiv \lambda_{\max} (R R^T) & 
	\gamma_R(w)  = q(R R^T,w)
		\in [ \gamma_{\min}^R,  \gamma_{\max}^R ] \equiv I_R,  \\[.6em]
		0 \le& 	\label{indK}
		\gamma_{\min}^K  \equiv \lambda_{\min} (K K^T),  & \gamma_{\max}^K  \equiv \lambda_{\max} (K K^T) & 
	\gamma_K(w)  = q(K K^T,w) \in [ \gamma_{\min}^K,  \gamma_{\max}^K ] \equiv I_K.
\end{array}
\]
Notice that zero can be a lower bound only for $\widehat{S}^{-1}D$, because $D$ can be singular, and $K K^T$, which is 
rank-deficient if $m < p$. From \eqref{EqRR}-\eqref{EqKK} and the definitions above, it follows immediately
\[ \gamma_S(w) = \gamma_R(w) + \gamma_D(w), \quad \gamma_X(w) = \gamma_K(w) + \gamma_E(w). \]
In the sequel, we will remove the argument $w$ whenever an indicator $\gamma_{*}$ will be used to light the notation.
Finally, we assume that $ 1\in [\gamma_{\min}^A, \gamma_{\max}^A$]. This assumption, which is very common in practice, will 
simplify some of the bounds that follow.

Exploiting \eqref{Eq31} yields:
\begin{align}
&\overline{A}x-\lambda x=(\lambda -1) R^Ty, \label{Eq32}\\
&Rx-\overline D y - (\lambda -1)K^Tz=-\lambda y, \label{Eq33}\\
&Ky + \overline E z =\lambda z. \label{Eq34}
\end{align}
Before characterizing both truly complex and real eigenvalues, we briefly analyze some particular cases 
where it is easy to show that $\lambda\in\mathbb{R}$.
Since $\widehat A, \widehat S$ and $\widehat X$ are SPD, and $B$ has full row rank
, then also $\overline{A}$ is SPD and $R$ has full row rank. 
Independently of the values of $m, p < n$, we can have:
  \begin{itemize}[leftmargin=20pt,topsep=5pt]
	  \item[]
		  \fbox {${x \ne 0 \land  y = 0}$} $\lambda \in [\gamma_{\min}^{A},\gamma_{\max}^{A}] \cap 
		  [\gamma_{\min}^{E},\gamma_{\max}^{E}] \in \R$ since it is either $\lambda=
		  \gamma_A(x)$
		  or $\lambda=
		  \gamma_E(z)$, if $z \ne 0$;
	  \item[]
		  \fbox {${x \ne 0 \land  Rx = 0}$} $\lambda \in [\gamma_{\min}^{A},\gamma_{\max}^{A}] \in \R$ because
		  premultiplying (\ref{Eq32}) by $x^T$ gives $\lambda = 
		  \gamma_A(x)$.
  \end{itemize}
To handle the remaining particular cases, we distinguish the two situations ($m \ge p$ and $m < p$). Recall that
$C$ has rank equal to $\min\{m,p\}$ and if $m \ge p$ ($m < p$), then $K (K^T)$ is full row rank.
\medskip

\noindent
\textbf{Case $\mathbf {m \ge p}$}.

  \begin{itemize}[leftmargin=20pt,topsep=5pt]
	  \item[]
		  \fbox {${x = 0}$} $\lambda = 1 \in \R$, because it must be $y\ne0$, otherwise this also would imply $z = 0$;
	  \item[]
		  \fbox {${x\ne 0, \ y \ne 0, \ K y= 0, \ z \ne 0}$} $\lambda \in [\gamma_{\min}^{E},\gamma_{\max}^{E}] \in \R$ 
		  since $\lambda = 
		  \gamma_E(z)$.
  \end{itemize}

\medskip

\noindent
\textbf{Case $\mathbf {m < p}$}.

  \begin{itemize}[leftmargin=20pt,topsep=5pt]
	  \item[]  
		  \fbox {${x = 0}$} $\lambda \in \{1\} \cup [\gamma_{\min}^{E},\gamma_{\max}^{E}] \in \R$, because it is
		  either $\lambda = 1$ from \eqref{Eq32}, or $K^T z = 0$ with $z \ne 0$, which implies $\lambda = 
		  \gamma_E(z)$;
	  \item[]
		  \fbox {${x \ne 0, \ K^T z= 0}$} $\lambda \in [\gamma_{\min}^{E},\gamma_{\max}^{E}] \in \R$ since it must
		  be $z \ne 0$ and $\lambda = 
		  \gamma_E(z)$.
\end{itemize}
In all other situations, $\lambda$ is generally complex. In such cases, the real and imaginary parts satisfy the
bounds that follow.

\subsection{Complex eigenvalues}
We first characterize the eigenvalues with non-zero imaginary part.

\begin{Theorem}\label{Th6}
Let $\lambda = a + i b$ be a complex eigenvalue of $\mathcal{A}{\mathcal{P}}^{-1}$, with $a$ the real part, $b$ the imaginary
part, $i$ the imaginary unit, and $(x; y; z)$ the corresponding eigenvector. Then, 
if $0\leq\gamma_{\min}^D<1$ any complex eigenvalue is such that
\begin{equation}
	\label{bound_on_modulus}
	|\lambda-1|  \le  \sqrt{ 1 - \rho_{\min}}, 
	\qquad \text{where} \quad
	\rho_{\min} = \frac{\gamma^A_{\min}\|x\|^2 + \gamma_{\min}^D \|y\|^2 + \gamma_{\min}^E \|z\|^2}
	{\|y\|^2 +\gamma^E_{\min}\|z\|^2}, 
\end{equation}
otherwise, all eigenvalues are real.
\end{Theorem}

\begin{proof}
With no loss of generality, we assume that $(x; y; z)$ is such that $\|x\|^2+\|y\|^2+\|z\|^2=1.$
By multiplying \eqref{Eq32} by $x^*$ on the left, the transposed conjugate of \eqref{Eq33} by $y$ on the right, and \eqref{Eq34} 
by $z^*$ on the left, we have:
\begin{align}
&x^*\overline{A}x-\lambda \|x\|^2=(\lambda -1)x^* R^Ty, \label{Eq44}\\
&x^*R^Ty-y^* \overline D y - (\bar{\lambda}-1)z^*Ky=-\bar{\lambda} \|y\|^2, \label{Eq45}\\
&z^*Ky + z^* \overline E z =\lambda \|z\|^2.\label{Eq46}
\end{align}
Eliminating $x^*R^Ty$ from \eqref{Eq45} and \eqref{Eq44} yields
\[ x^*\overline{A}x-\lambda \|x\|^2=(\lambda -1)\left(y^* \overline D y + (\bar{\lambda}-1)z^*Ky-\bar{\lambda} \|y\|^2\right), \]
which, recalling from \eqref{Eq46} that $z^*Ky = -z^* \overline E z +\lambda \|z\|^2$, after some algebra provides
\begin{equation}\label{Eq47}
	x^*\overline{A}x-\lambda\|x\|^2 = (y^* \overline D y -\bar \lambda \|y\|^2)  (\lambda -1)  + |\lambda-1|^2  \left(\lambda  \|z\|^2-z^* E z\right), 
\end{equation}
i.e., 
\begin{equation}\label{Eq47.1}
	(\gamma_A -\lambda)\|x\|^2  = (\gamma_D -\bar \lambda)  (\lambda -1) \|y\|^2 + |\lambda-1|^2  \left(\lambda -\gamma_E \right) \|z\|^2. 
\end{equation} 
	Let us 
	distinguish between the real and imaginary parts of equation \eqref{Eq47.1}:
\begin{align}
	(\gamma_A -a)\|x\|^2 &= (\gamma_D a - \gamma_D - |\lambda|^2 + a) \|y\|^2 
	+ |\lambda-1|^2  \left(a -\gamma_E \right) \|z\|^2, & &\text{(real part)}, \label{real} \\
	-b\|x\|^2 &= b \gamma_D \|y\|^2 - b \|y\|^2 + |\lambda-1|^2  b \|z\|^2, & &\text{(imaginary part)}. \label{imag}
\end{align}
If $\lambda$ is complex, then $b \ne 0$ and from (\ref{imag})  we obtain
	\begin{equation}
		\label{modulus}
  |\lambda -1 |^2 = \frac{\|y\|^2 (1 - \gamma_D)-\|x\|^2}{\|z\|^2}, 
		\end{equation}
	Equation \eqref{modulus} requires $0\leq \gamma_{\min}^D<1$, otherwise no complex eigenvalue occurs. 
	Defining the quantity
\begin{equation} 
	\label{rho}
\rho = 1- |\lambda-1|^2  = 2a - |\lambda|^2,
\end{equation}
and substituting (\ref{modulus}) in (\ref{real}) we have
	\begin{eqnarray*}
0 & = &  (\gamma_A -a)\|x\|^2  - (\gamma_D a - \gamma_D - |\lambda|^2 + a) \|y\|^2 - |\lambda-1|^2  \left(a -\gamma_E \right) \|z\|^2 \\
		& = &  (\gamma_A -a)\|x\|^2  - (\gamma_D a - \gamma_D - |\lambda|^2 + a) \|y\|^2 - (a -\gamma_E)
		\left( \|y\|^2 (1 - \gamma_D)-\|x\|^2\right ) \\
		& = &  (\gamma_A -\gamma_E)\|x\|^2  - (\gamma_D a - \gamma_D - |\lambda|^2 + a + a - \gamma_E - a \gamma_D  + \gamma_E \gamma_D) \|y\|^2  \\
		& = &  (\gamma_A -\gamma_E)\|x\|^2  + (\gamma_D + \gamma_E - \gamma_E \gamma_D - \rho) \|y\|^2, \\
		\end {eqnarray*}
from which
\begin{eqnarray*}\label{Eq51}
	\rho  &=& \gamma_D +  \gamma_E - \gamma_E \gamma_D +  (\gamma_A -\gamma_E)\frac{\|x\|^2}{\|y\|^2}  \\
	 &=& \frac{\gamma_D \|y\|^2 + \gamma_E  \|y\|^2- \gamma_E \gamma_D \|y\|^2 +  (\gamma_A -\gamma_E)\|x\|^2}{\|y\|^2}  \\
	 & = & \frac{\gamma_E(\|y\|^2 (1- \gamma_D) - \|x\|^2) + \gamma_A \|x\|^2 + \gamma_D \|y\|^2}{\|y\|^2}.
\end{eqnarray*}
		Then, by using again (\ref{modulus}), we get
		\begin{equation}
	 \rho =  \frac{\gamma_E\|z\|^2 |\lambda-1|^2 + \gamma_A \|x\|^2 + \gamma_D \|y\|^2}{\|y\|^2},
		\end{equation}
		which implies $\rho > 0$. Moreover, by using the definition of $\rho = 1-|\lambda-1|^2$ we have
		\begin{equation} \label{rho_gE}
	 \rho =  \frac{\gamma_A \|x\|^2 +\gamma_D \|y\|^2 + \gamma_E \|z\|^2}{\|y\|^2 + \gamma_E \|z\|^2}
			\ge  \frac{\gamma_{\min}^A \|x\|^2 +\gamma_{\min}^D \|y\|^2 + \gamma_E \|z\|^2}{\|y\|^2 + \gamma_E \|z\|^2}.
		\end{equation}
		We now show that the right-hand side of \eqref{rho_gE} is an increasing function in the variable $\gamma_E$. 
		We have:
		\begin{eqnarray*}
			h(t) &\equiv &
                        \frac{\gamma_{\min}^A \|x\|^2 +\gamma_{\min}^D \|y\|^2 + t \|z\|^2}{\|y\|^2 + t \|z\|^2} = \\
& = &	1 + \frac{\overbrace{\|x\|^2  +(\gamma_{\min}^D -1) \|y\|^2 + (\gamma_{\min}^A-1) \|x\|^2}^q}
			 {\|y\|^2 + t \|z\|^2} \equiv 1 + \frac{q}{\|y\|^2 + t \|z\|^2}.
		\end{eqnarray*} 
		Recalling equation (\ref{modulus}) and $\gamma^{\min}_A < 1$, we can conclude that $q < 0$, hence 
		$h(t)$ is an increasing function for $t\in[\gamma_{\min}^E,\gamma_{\max}^E]$. 
		By setting:
		\begin{equation*}
			\rho_{\min} = \min_{t\in I_E} h(t) \equiv 
			\frac{\gamma_{\min}^A \|x\|^2 +\gamma_{\min}^D \|y\|^2 + \gamma_{\min}^E \|z\|^2}{\|y\|^2 + \gamma_{\min}^E \|z\|^2} < 1,
		\end{equation*}
		from equation \eqref{rho_gE} we have $1-|\lambda-1|^2\geq\rho_{\min}$, which finally establishes the
		bound \eqref{bound_on_modulus}.
%
%
\end{proof}

In essence, Theorem \ref{Th6} states that the complex eigenvalues, if any, of the preconditioned matrix 
$\mathcal{A}\mathcal{P}^{-1}$ lie in a circle centered in 1 and with radius $\sqrt{1-\rho_{\min}}$, being $0<\rho_{\min}<1$.
This limits also the real part $a$ of any complex $\lambda$, such that $a\in[1-\sqrt{1-\rho_{\min}},1+\sqrt{1-\rho_{\min}}]$.
Moreover, by using equation \eqref{rho} $a$ can also be written as
\[ a = \frac{|\lambda|^2 - |\lambda-1|^2 + 1}{2} = \frac{|\lambda|^2 + \rho}{2} \ge \frac {\rho}{2} \ge \frac{\rho_{\min}}{2}, \]
which implies that $a$ cannot be arbitrarily small.
\textcolor{mygreen}{
\begin{Remark}
        A more practical bound on the radius of the circle containing complex eigenvalues can be obtained if we assume
	$\gamma_{\min}^D > 0$. Considering the function $h(t)$, we have
                \begin{eqnarray*} \rho_{\min}&=& \frac{\gamma^A_{\min}\|x\|^2 + \gamma_{\min}^D \|y\|^2 + \gamma_{\min}^E \|z\|^2}
			{\|y\|^2 +\gamma^E_{\min}\|z\|^2} \equiv  h(\gamma^{min}_E) \ge h(0) = 
\frac{\gamma^A_{\min}\|x\|^2 + \gamma_{\min}^D \|y\|^2}
        {\|y\|^2}  \ge \gamma_{\min}^D. \end{eqnarray*}
\end{Remark}
}

\subsection{Real eigenvalues}
In the following, we aim to characterize the real eigenvalues of the preconditioned matrix not contained in
$[\gamma^A_{\min},\gamma^A_{\max}]$ or $[\gamma^E_{\min},\gamma^E_{\max}]$. To this end, we first introduce three technical
lemmas that will be useful in the analysis
that follows.

\begin{Lemma}\cite[Lemma 1]{BergaNLAA10}\label{Le1}
	Let $\lambda \notin  [\gamma^A_{\min},\gamma^A_{\max}]$. Then, for an arbitrary vector $s\in\mathbb{R}^n$, $s \neq 0$, 
	there exists a vector $w\in\mathbb{R}^n$, $w \neq 0$, such that 
\begin{equation*}
\frac{s^T(\overline{A}-\lambda I)^{-1}s}{s^Ts}=\Big(\frac{w^T\overline{A}w}{w^Tw}-\lambda \Big)^{-1}=(\gamma_A-\lambda)^{-1}.
\end{equation*}
\end{Lemma}

\begin{Lemma}\label{Le2}
Let $p(x)$ be the polynomial 
	\[p(x)=x^3-a_2x^2+a_1x-a_0, \qquad a_0,a_1,a_2 >0, \] 
and $\alpha = \min\left \{a_2,\dfrac{a_0}{a_1}\right\}$, $\beta = \max\left \{a_2,\dfrac{a_0}{a_1}\right\}$. Then, 
	$p(x) < 0 \: \forall x \in (0,\alpha)$ and $p(x) > 0 \: \forall  x > \beta.$
\end{Lemma}

\begin{proof}
The statement of the lemma comes from observing that $p(x)$ is the sum of the term $x^3 - a_2x^2$,
which is negative in $(0, a_2)$ and positive for $x > a_2$, and of the term $a_1x - a_0$, which monotonically increases
with $x$ and changes sign for $x = \dfrac{a_0}{a_1}.$
\end{proof}



\begin{Lemma}\label{Le4}
	Let $\zeta \in \mathbb{R}$ be either  $0 < \zeta < \min\left\{\gamma_{\min}^A, \dfrac{\gamma_{\min}^R}{\gamma_{\max}^A+\gamma_{\min}^R+ \gamma^D_{\max}}\right\}$ or $\zeta \ge \gamma_{\max}^A+\gamma_{\max}^S$.
	Then, all the eigenvalues of the symmetric matrix $Z\in\mathbb{R}^{m\times m}$
\begin{equation}\label{Eq622}
Z(\zeta) = (1-\zeta )  R (\zeta I - \overline A)^{-1} R^T - \overline D + \zeta  I,
\end{equation}
are either positive or negative.
\end{Lemma}

\begin{proof}
Let $w\in\mathbb{R}^m$ be an arbitrary non-zero vector. 
The Rayleigh quotient $q(w,\zeta)$ associated to the matrix $Z(\zeta)$ in \eqref{Eq622} reads:
\begin{equation}\label{RayZ} 
	q(w,\zeta) = \frac{w^T Z w}{w^T w} = (1-\zeta) \frac{w^T R (\zeta I - \overline A)^{-1} R^T w}{w^T w}- 
	\frac{w^T \overline D w}{w^T w} +\zeta.  
\end{equation}
By setting $z=R^Tw$ and recalling that $\gamma_S=\gamma_R+\gamma_D$, we have:
	\begin{eqnarray}\label{RayZ2}
		q(w,\zeta) &=& (1-\zeta) \frac{z^T (\zeta I - \overline A)^{-1} z}{z^T z} \frac{w^T R R^T w}{w^T w} - 
		\frac{w^T \overline D w}{w^T w} + \zeta   \nonumber \\
	&=& (1-\zeta) \frac{z^T (\zeta I - \overline A)^{-1} z}{z^T z} (\gamma_S - \gamma_D)- \gamma_D + \zeta,   
	\end{eqnarray}
	which yields by applying Lemma \ref{Le1}:
	\begin{equation}\label{Eq63}
		q(w,\zeta) =    \frac{1-\zeta}{\zeta - \gamma_A} \gamma_S -\gamma_D \frac{1-\gamma_A}{\zeta - \gamma_A}+ \zeta 
\end{equation}
	Now, $q(w,\zeta) = 0$ from equation \eqref{Eq63} implies
\begin{equation}\label{Eq64}
	h(\zeta) = \zeta^2 - (\gamma_A+\gamma_S) \zeta +\gamma_S + \gamma_D (\gamma_A -1) = 0.
\end{equation}
	The roots $h_-$ and $h_+$ of \eqref{Eq64} are
	\begin{eqnarray} 
		h_- &=& \frac{\gamma_A+\gamma_S - 
		\sqrt{(\gamma_A+\gamma_S)^2 - 4\left(\gamma_R + \gamma_D \gamma_A\right)}}{2} = 
		\frac{2(\gamma_R + \gamma_D \gamma_A)}{\gamma_A+\gamma_S + 
		 \sqrt{(\gamma_A+\gamma_S)^2 - 4\left(\gamma_R + \gamma_D \gamma_A\right)}}, \label{lowerh} \\
		 h_+ &=& \frac{\gamma_A+\gamma_S + 
		\sqrt{(\gamma_A+\gamma_S)^2 - 4\left(\gamma_R + \gamma_D \gamma_A\right)}}{2}. 
		\label{upperh}
	\end{eqnarray} 
Since $\gamma_A+\gamma_S + \sqrt{(\gamma_A+\gamma_S)^2 - 4\left(\gamma_R + \gamma_D \gamma_A\right)} < 2(\gamma_A+\gamma_S)$,
	we have the following bounds for \eqref{lowerh} and \eqref{upperh}:
	\begin{eqnarray}
		h_- &\ge& \frac{\gamma_R + \gamma_D \gamma_A}{\gamma_A+\gamma_S}
                 \ge \frac{\gamma_R}{\gamma_A+\gamma_R + \gamma_D}  \ge \frac{\gamma^R_{\min}}{\gamma^A_{\max} +
		 \gamma^R_{\min} + \gamma^D_{\max}}, \label{lowerh2} \\
		 h_+ &\le& \gamma_A + \gamma_S \le \gamma^A_{\max} + \gamma^S_{\max}. \label{upperh2}
	\end{eqnarray}
        Hence, the Rayleigh quotient associated to $Z(\zeta)$ cannot be zero under the hypotheses on $\zeta$.
\end{proof}

We are now able to bound the real eigenvalues of the preconditioned matrix $\mathcal{A}{\mathcal{P}}^{-1}$. 

\begin{Theorem}
	\label{Theorem_real}
	Let $\lambda$ be a real eigenvalue of $\mathcal{A}\mathcal{P}^{-1}$ with \[\lambda \notin 
	\left[ \min\left\{\gamma_{\min}^A, \dfrac{\gamma_{\min}^R}{\gamma_{\max}^A+\gamma_{\min}^R+ \gamma^D_{\max}}\right\}, \gamma_{\max}^A+\gamma_{\max}^S\right]. \]
	Then, $\lambda$ 
satisfies
	\begin {equation}
		\label{CubicPoly}
	\lambda^3-(\gamma_A+\gamma_{S}+\gamma_{X})\lambda^2{\red +}(\gamma_A \gamma_X + \gamma_K + \gamma_E \gamma_S + \gamma_D \gamma_A + \gamma_R)\lambda 
		{\red -} \left(\gamma_A \gamma_K + \gamma_E \gamma_A \gamma_D + \gamma_E \gamma_R\right) = 0,
\end{equation}
and the following synthetic bound holds:
	\begin {equation}  
		\label{boundsKynot0}
	 \min\left\{\gamma_{\min}^E, \gamma_{\min}^A, \dfrac{\gamma_{\min}^R}{\gamma_{\max}^A+\gamma_{\min}^R+ \gamma^D_{\max}}\right\}
 \le \lambda  \le
\gamma_{\max}^A+\gamma_{\max}^{S}+\gamma_{\max}^{X}.
\end{equation}
\end{Theorem}

\begin{proof}
	The equation  \eqref{Eq32} can be written as
\begin{equation}\label{Eq65}
x = (1-\lambda) (\lambda I -\overline{A})^{-1} R^T y.
\end{equation}
	Introducing $x$ of \eqref{Eq65} into \eqref{Eq33}, we have
\begin{equation}\label{Eq66}
Z(\lambda)y = (\lambda-1) K^T z,
\end{equation}
where $Z(\lambda)=(1-\lambda) R(\lambda I -\overline{A})^{-1} R^T - \overline D  + \lambda I$.
        The hypotheses on $\lambda$ allow us to use Lemma \ref{Le4}, which guarantees the definiteness of the matrix function  $Z(\lambda)$. 
		Hence, obtaining
	$y = (\lambda - 1)Z(\lambda)^{-1} K^T z$ from \eqref{Eq66} and substituting it in \eqref{Eq34} yields
\begin{equation}\label{Eq67}
\left[K(\lambda-1)Z(\lambda)^{-1} K ^T + \overline E -\lambda I \right]z = 0.
\end{equation}
Premultiplying by $z^T$on the left and dividing by $z^Tz$ yields
\begin{equation}\label{Eq68}
\frac{z^TK (\lambda-1) Z(\lambda)^{-1} K ^T z}{z^Tz}+ \gamma_E -\lambda=0.
\end{equation}
By setting $w=K^Tz$, we have
\begin{equation}\label{Eq68_2}
	(\lambda-1) \frac{w^TZ(\lambda)^{-1}w}{w^Tw} \frac{z^T K K^T z}{z^T z}+ \gamma_E -\lambda =0,
\end{equation}
	which becomes:
\begin{equation}\label{Eq68_1}
	(\lambda-1) \frac{u^T u}{u^T Z(\lambda) u} \frac{z^T K K^T z}{z^T z}+ \gamma_E -\lambda =0,
\end{equation}
by setting $u = Z(\lambda)^{-1/2}w$.
	Using now \eqref{Eq63} in Lemma \ref{Le4}, we get
\begin{equation}
  \dfrac{\lambda-1}{\dfrac{1-\lambda}{\lambda - \gamma_A} \gamma_S -\gamma_D \dfrac{1-\gamma_A}{\lambda - \gamma_A}+ \lambda} \gamma_K + \gamma_E - \lambda = 0,
\end{equation}
and after simple algebra, we are left with:
	\begin{equation}
		\left[\lambda^2 - (1 + \gamma_A)\lambda + \gamma_A \right] \gamma_K + (\gamma_E - \lambda) \left[
		\lambda^2 - (\gamma_A+\gamma_S) \lambda +\gamma_S + \gamma_D (\gamma_A -1)\right] = 0,
	\label{eq:intermediate}
	\end{equation}
 which can be rewritten as:
	\begin{equation}
		q(\lambda) = s(\lambda) \gamma_K + 
		(\gamma_E - \lambda) h(\lambda)  = 0,
		\label{before}
	\end{equation}
	where $h(\lambda)$ is the same polynomial as in \eqref{Eq64}.
%
	We now observe that $q(0) =\gamma_A \gamma_K + \gamma_E(\gamma_R + \gamma_D \gamma_A) > 0$. Then,
	denoting by:
	\begin{equation*}
		\lambda_{-}^{UB} = \frac{\gamma^R_{\min}}{\gamma^A_{\max}+\gamma^R_{\min}+\gamma^D_{\max}} < 1,
	\end{equation*}
	we notice that $s(\lambda)$, $h(\lambda)$ and $(\gamma_E-\lambda)$ are all positive for any $\lambda\in
	[0,\min\{\gamma^E_{\min},\gamma^A_{\min},\lambda_{-}^{UB}\}]$, thus proving the left inequality in \eqref{boundsKynot0}.

	Equation (\ref{before}) can be more explicitly rewritten as
\begin{equation}\label{Eq61}
	\lambda^3-a_2 \lambda^2+a_1\lambda-a_0=0, 
\end{equation}
where
	\begin{eqnarray}
		\label{coefficients}
		a_2 & = & \gamma_X + \gamma_S + \gamma_A \\
		a_1 & = & \gamma_A \gamma_X + \gamma_K + \gamma_E \gamma_S + \gamma_D \gamma_A + \gamma_R\\
		a_0 & = & \gamma_A \gamma_K + \gamma_E \gamma_A \gamma_D + \gamma_E \gamma_R.
	\end{eqnarray}
	Applying Lemma \ref{Le2} to this cubic polynomial we have
	\[ \alpha = \frac{\gamma_A \gamma_K + \gamma_E \gamma_A \gamma_D + \gamma_E \gamma_R}{\gamma_A \gamma_X + \gamma_K + \gamma_E \gamma_S + \gamma_D \gamma_A + \gamma_R}, \ 
	\beta = \gamma_A+\gamma_{S}+\gamma_{X} .\]
	In this case, it is easy to verify that $\alpha < \beta$, from which we have that
	\[ \lambda \le \gamma_{\max}^A+\gamma_{\max}^{S}+\gamma_{\max}^{X}, \]
	which is the right inequality of \eqref{boundsKynot0}.
\end{proof}

\section{Eigenvalue bounds for the block diagonal preconditioner}
\label{sec:eig_diag}
For the sake of comparison,
we consider now as a preconditioner for the double saddle-point linear system \eqref{Eq1} the block diagonal SPD matrix:
\[
        {\mathcal{P}_D}=\begin{bmatrix} \widehat A &0&0\\0&\widehat S &0\\0&0& \widehat X\end{bmatrix},
		\]
		already introduced in \eqref{Eq30}. The use of $\mathcal{P}_D$ as a preconditioner instead of $\mathcal{P}$
		allows for employing MINRES instead of GMRES.
		The eigenvalues of the preconditioned matrix $\mathcal{A}\mathcal{P}_D^{-1}$ are the same as those of
		$\mathcal{P}_D^{-1/2} \mathcal{A}\mathcal{P}_D^{-1/2}$. Recalling equation \eqref{Eq31}, the corresponding
		eigenproblem reads:
\begin{align}
&\overline{A}x+R^Ty=\lambda x, \label{EqD32.11}\\
&Rx-\overline D y + K^Tz= \lambda y, \label{EqD32.12}\\
&Ky + \overline E z =\lambda z. \label{EqD32.13}
\end{align}
	Assuming $\lambda \not \in [\gamma_A^{\min}, \gamma_A^{\max}] $,
	equation \eqref{EqD32.11} can be written as
\begin{equation}\label{EqD65}
x = (\lambda I - \overline A)^{-1} R^T y,
\end{equation}
	which replaced into \eqref{EqD32.12} yields
\begin{equation}
	\left( - R (\lambda I - \overline A)^{-1} R^\top  + \overline D +\lambda I \right) y = K^T z.
	\label{WyKz}
\end{equation}
Setting $W(\lambda)=-R(\lambda I-\overline A)^{-1}R^T+\overline D+\lambda I$, 
we now investigate for which values of $\lambda$ yield 0 is not included in the interval $[\lambda_{\min}(W(\lambda)),\lambda_{\max}(W(\lambda))]$. 
	Applying the lemma \ref{Le1}, 
	the Rayleigh quotient $q(W(\lambda),s)$ reads:
	\begin{equation}
		\label{yZyDIAG}
		q(W(\lambda),s) =
		\frac{s^T W(\lambda) s}{s^T s} =
		-\frac{s^T R (\lambda I-\overline A)^{-1} R^T s}{s^T s} + \gamma_D + \lambda =
		-\frac{\gamma_R}{\lambda-\gamma_A} + \gamma_D + \lambda =
	\frac{p(\lambda)}{\lambda - \gamma_A},
	\end{equation}
	with
	\[p(\lambda) = \lambda^2 - (\gamma_A - \gamma_D) \lambda- \gamma_R -\gamma_A \gamma_D.\]
	Denote as $\lambda_-, \lambda_+$ the two roots of $p$:
	\begin{eqnarray}
		\lambda_+ & = & \frac{\gamma_A-\gamma_D+\sqrt{(\gamma_A+\gamma_D)^2+4\gamma_R}}{2}, \nonumber \\
		\lambda_- & = & \frac{\gamma_A-\gamma_D-\sqrt{(\gamma_A+\gamma_D)^2+4\gamma_R}}{2} 
	\end{eqnarray}
	Since $p(\gamma_A + \gamma_R) = \gamma_R(\gamma_R + \gamma_A + \gamma_D) > 0$ and $p(\gamma_A) = -\gamma_R  < 0$,
	we have that $\gamma_A<\lambda_+<\gamma_A+\gamma_R$. 
	\textcolor{red}{Moreover, as $\lambda_+$ is increasing in $\gamma_A, \gamma_R$ and decreasing in $\gamma_D$, we have the following,
	possibly tighter, bound:
	\[ \lambda_+ \le \frac{\gamma_{\max}^A+\sqrt{(\gamma_{\max}^A)^2+4\gamma_{\max}^R}}{2} \le \gamma_{\max}^A + \sqrt{\gamma_{\max}^R}.\]
	Hence, if $\lambda>0$, $p(\lambda)$ can be zero only if
	$\lambda\in[\gamma^A_{\min},\gamma^A_{\max}+\sqrt{\gamma^R_{\max}}]$. 
	The negative root, $\lambda_-$, is increasing in $\gamma_A$ and
        decereasing in $\gamma_D, \gamma_R$ therefore,
        \[ \lambda_-(\gamma_A, \gamma_R, \gamma_D)  \ge \lambda_-(0, \gamma_R, \gamma_D) = \frac{-\gamma_D-\sqrt{\gamma_D^2+4\gamma_R}}{2}
        \ge  -\gamma_D - \sqrt{\gamma_R}
	\ge  -\gamma_{\max}^D - \sqrt{\gamma_{\max}^R}. \]
        To bound $\lambda_-$ from above, we bound its modulus from below:
        \[ |\lambda_-|  =  \frac{\gamma_R +\gamma_A \gamma_D}{\lambda_+}  > \frac{\gamma_R +\gamma_A \gamma_D}{\gamma_A + \gamma_R}
                = 1 + \frac{\gamma_A(\gamma_D-1)}{\gamma_A + \gamma_R}.
                 \]
		 If $\gamma_D \ge 1$ then $ |\lambda_-|  \ge 1$, otherwise, the last expression can be conveniently rewritten as  
		 \[ \frac{\gamma_R +\gamma_A \gamma_D}{\gamma_A + \gamma_R} = 
		  \frac{\gamma_S +(\gamma_A-1) \gamma_D}{\gamma_A + \gamma_S - \gamma_D},\] 
		  which, upon the assumption $\gamma_D < 1$,  can be shown to be decreasing in $\gamma_A$, so that
		  \begin{eqnarray*} \frac{\gamma_S +(\gamma_A-1) \gamma_D}{\gamma_A + \gamma_S - \gamma_D} &\ge&
			  \frac{\gamma_S +(\gamma_{\max}^A-1) \gamma_D}{\gamma_{\max}^A + \gamma_S - \gamma_D} \ge \ \text{(being
			  the previous increasing in both $\gamma_D$ and $\gamma_S$ \ since $\gamma_{\max}^A > 1$)}, \\
			  & \ge & \frac{\gamma_{\min}^S +(\gamma_{\max}^A-1) \gamma_{\min}^D}{\gamma_{\max}^A + \gamma_{\min}^S - \gamma_{\min}^D}\ge
		  \frac{\gamma_{\min}^S }{\gamma_{\max}^A + \gamma_{\min}^S}.\end{eqnarray*} 
		  Summarizing, as this last quantity is less than one,
	 $\lambda_- \le -\dfrac{\gamma_{\min}^S}{\gamma_{\max}^A + \gamma_{\min}^S}$.
	Therefore, if $\lambda$ is not included in the set:
	\[ I_2 = \left[-\gamma_{\max}^D - \sqrt{\gamma_{\max}^R}, -\dfrac{\gamma_{\min}^S}{\gamma_{\max}^A + \gamma_{\min}^S}\right ] \cup \left[\gamma_{\min}^A , \gamma_{\max}^A+ \sqrt{\gamma_{\max}^R}\right], \]
	} then
$0 \not \in [\lambda_{\min}(W(\lambda)),\lambda_{\max}(W(\lambda))]$ and the eigenvalues of $W(\lambda)$ are either all
positive or all negative. 
Under this condition, we can then proceed as in the 
proof of Theorem \ref{Theorem_real}. Introducing $y=W(\lambda)^{-1}K^Tz$ from \eqref{WyKz} into \eqref{EqD32.13} we 
obtain:
\begin{equation*}
	K W(\lambda)^{-1} K^T z + \overline E z - \lambda z = 0.
\end{equation*}
By pre-multiplying both sides by $z^T$ and dividing by $z^Tz$, and recalling the expression of the Rayleigh quotient
for $W(\lambda)$ in \eqref{yZyDIAG}, we have:
\begin{equation*}
	\frac{\lambda - \gamma_A}{p(\lambda)} \gamma_K + \gamma_E - \lambda = 0,
\end{equation*}
from which we can conclude that the eigenvalues $\lambda$ of $\mathcal{A}\mathcal{P}_D^{-1}$ satisfy the third-degree polynomial:
\begin{equation} 
	\pi(\lambda) = 
	\lambda^3 - \underbrace{(\gamma_A + \gamma_E - \gamma_D}_{\text{\normalsize $\widehat a_2$}}) \lambda^2 - 
\underbrace{(\gamma_R + \gamma_K + \gamma_E \gamma_D + \gamma_A \gamma_D - \gamma_E \gamma_A)}_{\text{\normalsize $\widehat a_1$}} \lambda +
	\underbrace{\gamma_A \gamma_K  + \gamma_E \gamma_R + \gamma_E \gamma_A \gamma_D}_{\text{\normalsize $\widehat a_0$}}.
	\label{pilambda}
\end{equation}
There are one negative ($\mu_a$) and two positive ($\mu_b<\mu_c$) roots  of $\pi(\lambda)$, as also proved in \cite[Corollary 2.1]{Bradley}.
The next theorem provides bounds on the eigenvalues of the preconditioned matrix $\mathcal{A}{\mathcal{P}_D}^{-1}$.
\begin{Theorem}
        \label{Theorem_real_D}
        The eigenvalues of the preconditioned matrix $\mathcal{A}{\mathcal{P}_D}^{-1}$ belong to $I_- \cup I_+$, where
        \begin {eqnarray*}\
	I_- & = & \left[-\gamma_{\max}^D - \sqrt {\gamma_{\max}^R+ \gamma_{\max}^K}, - \textcolor{red}{\frac{\gamma_{\min}^S}{\gamma_{\max}^A + \gamma_{\min}^S}} \right] \\
        I_+ &=& \left[\min\{\gamma_{\min}^E, \gamma_{\min}^A \},
        \max\left\{\gamma_{\max}^A + \gamma_{\max}^E + \sqrt{\gamma_{\max}^K + \gamma_{\max}^R},
        \sqrt{\gamma_{\max}^R + \gamma_{\max}^K + \gamma_{\max}^D (\gamma_{\max}^E + \gamma_{\max}^A)}\right \} \right]
        \end {eqnarray*}
\end{Theorem}

\begin{proof}
	The proof is subdivided into 4 sections, corresponding to the limits for $I_+$ and $I_-$, respectively. 
\begin{enumerate}
\item Lower bound for $\mu_b$.
	Notice that $\pi(\gamma_E)  =  \gamma_K (\gamma_A - \gamma_E)$, and $\pi(\gamma_A)  =  -(\gamma_R + \gamma_A \gamma_D) (\gamma_A - \gamma_E)$. Hence, $\pi(\gamma_A) \pi(\gamma_E) < 0$, which means that $\mu_b\in[\gamma_A,\gamma_E]$ or $\mu_b\in[\gamma_E,\gamma_A]$ depending whether $\gamma_A$ is larger or smaller than $\gamma_E$. It follows that
		\[ 0 < \min\{\gamma_{\min}^E, \gamma_{\min}^A \} \le \mu_b.\]
\item Upper bound for $\mu_c$.
	The polynomial $\pi(\lambda)$ can be written also as:
\[ \pi(\lambda) = \lambda \varphi(\lambda) + \widehat a_0,\]
where
\[ \varphi(\lambda) =  \lambda^2 - \left(\gamma_A + \gamma_E - \gamma_D\right) \lambda - \left(\gamma_R + \gamma_K + \gamma_E \gamma_D + \gamma_A \gamma_D - \gamma_E \gamma_A\right). \]
		Since $\widehat{a}_0>0$, an upper bound for the largest root of $\varphi$ is also a bound for $\mu_c$.
Direct computation shows that
\begin{eqnarray*} \varphi\left(\gamma_A + \gamma_E + \sqrt{\gamma_K + \gamma_R} \right) &= &
        (\gamma_A + \gamma_E) ^2 + 2(\gamma_A + \gamma_E)  \sqrt{\gamma_K + \gamma_R} + \gamma_K + \gamma_R\\
        &-& (\gamma_A + \gamma_E)^2 - \gamma_D (\gamma_A + \gamma_E) - (\gamma_A + \gamma_E + \gamma_D) \sqrt{\gamma_K + \gamma_R}\\
        &-& (\gamma_R + \gamma_K + \gamma_E \gamma_D + \gamma_A \gamma_D - \gamma_E \gamma_A). \\
        &=& (\gamma_E + \gamma_A - \gamma_D) \sqrt{\gamma_K + \gamma_R}  + \gamma_E \gamma_A,
        \end{eqnarray*}
        which is positive for $\gamma_D < \gamma_E + \gamma_A$.  If $\gamma_D \ge \gamma_E + \gamma_A$, positivity
        of $\varphi$ is guaranteed under the condition:
        \[ \lambda > \sqrt{\gamma_R + \gamma_K + \gamma_D (\gamma_E + \gamma_A)}. \]
        Summarizing, we have \[\mu_c \le \max\left\{\gamma_{\max}^A + \gamma_{\max}^E + \sqrt{\gamma_{\max}^K + \gamma_{\max}^R},
        \sqrt{\gamma_{\max}^R + \gamma_{\max}^K + \gamma_{\max}^D (\gamma_{\max}^E + \gamma_{\max}^A)}\right \}. \]
\item Upper bound for $\mu_a$.
        Since $\pi(\lambda_-) >  0$,  we immediately get an upper bound for $\mu_a$ given by
			\[ \mu_a \le h_- \le \textcolor{red}{- \frac{\gamma_{\min}^S}{\gamma_{\max}^A + \gamma_{\min}^S}}. \]
\item Lower bound for $\mu_a$.
        {
		Let us write $\pi(\lambda)$ as
                \[ \pi(\lambda) = \gamma_A q_A(\lambda) + \gamma_E q_E(\lambda) + r(\lambda), \]
                where
                \begin{eqnarray}
                        q_A(\lambda) &=&  \gamma_K - (\lambda - \gamma_E)(\lambda + \gamma_D) \label{qA}, \\
			q_E(\lambda) &=&  -p(\lambda), \label{qE} \\
			r(\lambda) &=& \lambda(\lambda^2  + \gamma _D \lambda  - \gamma_R - \gamma_K) \label{rl} .
                \end{eqnarray}
                Since $\mu_A$ is a root of $\pi$, the following relationship must hold:
		\begin{equation}
			\label{ma-gE}
                \mu_a - \gamma_E = \frac{\gamma_K(\mu_A - \gamma_A)}{p(\mu_a)}.
		\end{equation}
		Substituting \eqref{ma-gE} into \eqref{qA} and rearranging yields
\[ q_A(\mu_a)  = -\frac{\gamma_R \gamma_K}{p(\mu_A)}. \]
		Equation \eqref{ma-gE} also tells us that $p(\mu_A) > 0$, because $\mu_a$ is negative. As a consequence,
both $q_A(\mu_A)$ and $q_E(\mu_A)$ are negative, hence it must be $r(\mu_A) >  \pi(\mu_A) = 0$ and therefore
the negative root of $r(\lambda)$ is a lower bound for $\mu_a$. 
The negative root of $r(\lambda)$ is
\begin{eqnarray*} s_- &=& \frac{-\gamma_D - \sqrt{\gamma_D^2 + 4 (\gamma_R+ \gamma_K)}}{2}
        \ge -\gamma_D - \sqrt {\gamma_R+ \gamma_K}
	        \ge -\gamma_{\max}^D - \sqrt {\gamma_{\max}^R+ \gamma_{\max}^K}.
        \end{eqnarray*}
                }
        \end{enumerate}
\end{proof}

\section{Mixed form of Biot's poroelasticity equations}
\label{SecBiot}
The classical formulation of linear poroelasticity couples Darcy's flow of a single-phase fluid with the quasi-static
mechanical deformation of an elastic and fully saturated porous medium.
In their mixed form, the governing equations consist of a conservation law of linear momentum, a conservation law of 
fluid mass, and Darcy's law connecting the fluid velocity to the variation of the hydraulic head.

Let $\Omega \subset \mathbb{R}^d$ ($d=2,3$) be the domain occupied by the porous medium with $\Gamma$ its 
Lipschitz-continuous boundary, which
is decomposed as $\Gamma = \overline{\Gamma_{u} \cup \Gamma_{\sigma}} = \overline{\Gamma_p \cup \Gamma_{q}}$, with 
$\Gamma_u \cap \Gamma_{\sigma} = \Gamma_p \cap \Gamma_{q} = \emptyset$.
The strong form of the initial boundary value problem (IBVP) consists of finding the displacement 
$\tensorOne{u}:\overline{\Omega}\times]0,t_{\max}[ \rightarrow \mathbb{R}^d$, Darcy's velocity 
$\tensorOne{q}:\overline{\Omega}\times]0,t_{\max}[ \rightarrow \mathbb{R}^d$, and the excess pore pressure 
$p:\overline{\Omega}\times]0,t_{\max}[ \rightarrow \mathbb{R}$ such that:
\begin{subequations}
\begin{align}
  \nabla \cdot \left( \tensorFour{C}_{dr} : \nabla^{s} \tensorOne{u} - b p \tensorTwo{1} \right)
  &=
  \tensorOne{0} & &\mbox{ in } \Omega \times ]0,t_{\max}[
  & &\mbox{(equilibrium)}, \label{momentumBalanceS}\\
  \mu \tensorTwo{\kappa}^{-1} \cdot \tensorOne{q} + \nabla p
  &= 
  \tensorOne{0} & &\mbox{ in } \Omega \times ]0,t_{\max}[
  & &\mbox{(Darcy's law)}, \label{darcyS}  \\	
  b \nabla \cdot \dot{\tensorOne{u}} + S_{\epsilon} \dot{p} + \nabla \cdot \tensorOne{q} &=
  f & &\mbox{ in } \Omega \times ]0,t_{\max}[
  & &\mbox{(continuity)}, \label{massBalanceS}
\end{align}\label{eq:IBVP}\null
\end{subequations}
where $\tensorFour{C}_{dr}$ is the rank-four elasticity tensor, $\nabla^{s}$ is the symmetric gradient operator, 
$b$ is the Biot coefficient, and $\tensorTwo{1}$ is the rank-two identity tensor; $\mu$ and $\tensorTwo{\kappa}$ are 
the fluid viscosity and the rank-two permeability tensor, respectively; $S_{\epsilon}$ is the constrained specific 
storage coefficient, i.e., the reciprocal of Biot's modulus, and $f$ the fluid source term. 
The following set of boundary and initial conditions completes the formulation:
\begin{subequations}
\begin{align}
  \tensorOne{u} &= \bar{\tensorOne{u}}
  & &\mbox{ on } \Gamma_u \times ]0,t_{\max}[, & \label{momentumBalanceS_DIR}\\
  \left( \tensorFour{C}_{dr} : \nabla^{s} \tensorOne{u} - b p \tensorTwo{1} \right) \cdot \tensorOne{n} &= 
  \bar{\tensorOne{t}}
  & &\mbox{ on } \Gamma_{\sigma} \times ]0,t_{\max}[, & \label{momentumBalanceS_NEU}\\
  \tensorOne{q} \cdot \tensorOne{n} &= \bar{q}
  & &\mbox{ on } \Gamma_{q} \times ]0,t_{\max}[, &\label{massBalanceS_NEU}\\    
  p &=\bar{p}
  & &\mbox{ on } \Gamma_p \times ]0,t_{\max}[, & \label{massBalanceS_DIR}\\
	\tensorOne{u}(\tensorOne{x}, 0) &=\tensorOne{u}_0
  & &\mbox{ } \tensorOne{x} \in \overline{\Omega}, &  \label{momentumBalanceS_IC}\\
  p(\tensorOne{x}, 0) &=p_0
  & &\mbox{ } \tensorOne{x} \in \overline{\Omega}, &  \label{massBalanceS_IC}
\end{align}\label{eq:IBVP_b}\null
\end{subequations}
where $\bar{\tensorOne{u}}$, $\bar{\tensorOne{t}}$, $\bar{q}$, and $\bar{p}$ are the prescribed boundary displacements, 
tractions, Darcy velocity, and excess pore pressure, respectively, whereas $p_0$ is the initial excess pore pressure.
By introducing the spaces:
\begin{subequations}
\begin{align}
	\boldsymbol{\mathcal{U}} &= \{ \tensorOne{u} \in [H^1(\Omega)]^d \ \ | \ \
  \tensorOne{u}|_{\Gamma_u}=\bar{\tensorOne{u}} \}, &
	\boldsymbol{\mathcal{U}}_0 &= \{ \tensorOne{u} \in [H^1(\Omega)]^d \ \ | \ \
  \tensorOne{u}|_{\Gamma_u}=\tensorOne{0} \}, \label{eq:spaceH1_0}  \\
	\boldsymbol{\mathcal{Q}} &= \{\tensorOne{q} \in [H(\text{div};\Omega)]^d \ \ | \ \
  \tensorOne{q} \cdot \tensorOne{n}|_{\Gamma_q}=\bar{q} \}, &
	\boldsymbol{\mathcal{Q}}_0 &= \{\tensorOne{q} \in [H(\text{div};\Omega)]^d \ \ | \ \
  \tensorOne{q} \cdot \tensorOne{n}|_{\Gamma_q}=0 \}, \label{eq:spaceHdiv0}
\end{align}\label{eq:functionSpaces}
\end{subequations}
the weak form of the IBVP \eqref{eq:IBVP} reads: find $\{ \tensorOne{u}(t),\tensorOne{q}(t), p(t) \} \in 
\boldsymbol{\mathcal{U}} \times \boldsymbol{\mathcal{Q}} \times L^2(\Omega)$ such that $\forall t \in ]0,t_{\max}[$:
\begin{subequations}
\begin{align}
  &{(\nabla^s \tensorTwo{\eta}, \tensorFour{C}_{\text{dr}}:\nabla^s \tensorOne{u})}_{\Omega}-{(\text{div} \; \tensorTwo{\eta},bp)}_{\Omega}
  ={ ( \tensorTwo{\eta}, \bar{\tensorOne{t}} ) }_{\Gamma_\sigma} &&\forall \tensorTwo{\eta} \in \boldsymbol{\mathcal{U}}_0 , \label{momentumBalanceW}\\
  &{(\tensorTwo{\phi}, \mu \tensorTwo{\kappa}^{-1} \cdot \tensorOne{q})}_{\Omega} - {(\text{div} \; \tensorTwo{\phi}, p)}_{\Omega}
  = - {(\tensorTwo{\phi} \cdot \tensorOne{n}, \bar{p})}_{\Gamma_p}  &&\forall \tensorTwo{\phi} \in \boldsymbol{\mathcal{Q}}_0, \label{darcyW} \\
  &{(\chi, b \; \text{div} \; \dot{\tensorOne{u}})}_{\Omega} + {(\chi, \text{div} \; \tensorOne{q})}_{\Omega} + {(\chi, S_{\epsilon} \dot{p})}_{\Omega} =
  {(\chi,f)}_{\Omega} &&\forall \chi \in L^2(\Omega).\label{massBalanceW}
\end{align}\label{eq:IBVP_cW}\null
\end{subequations}
The final fully-discrete form of the continuous problem \eqref{eq:IBVP_cW} is obtained by introducing a non-overlapping partition
$\mathcal{T}_h$ of the spatial domain $\Omega$ with the associated finite dimensional counterparts $\boldsymbol{\mathcal{U}}^h$ 
and $\boldsymbol{\mathcal{Q}}^h$ of the spaces \eqref{eq:functionSpaces}, and $\mathcal{P}^h$ of $L^2(\Omega)$. Then,
discretization in time is performed by a backward Euler scheme.

\subsection{Mixed Finite Element (MFE) discretization}
Let $\{\tensorTwo{\eta}\}$, $\{\tensorTwo{\phi}\}$, and $\{\chi\}$ be the set of basis function of the discrete spaces
$\boldsymbol{\mathcal{U}}^h$, $\boldsymbol{\mathcal{Q}}^h$, and $\mathcal{P}^h$, respectively. For example, a very
common choice is based on the triplet $\mathbb{Q}_1-\mathbb{RT}_0-\mathbb{P}_0$, with $\mathbb{Q}_1$ the space
of bi- or tri-linear piecewise polynomials, according to the size $d$ of the domain $\Omega$, $\mathbb{RT}_0$ the
lowest-order Raviart-Thomas space, and $\mathbb{P}_0$ the space of piecewise constant functions. Denoting by $u_t\in\mathbb{R}^n$,
$p_t\in\mathbb{R}^m$, and $q_t\in\mathbb{R}^p$ are the vectors of discrete displacement, pressure, and velocity values at the time
$t$, the fully-discrete form of \eqref{eq:IBVP_cW} reads:
\begin{equation}
	\begin{bmatrix}
		A_{uu} & A_{up} & 0 \\
		A_{pu} & A_{pp} & \Delta t A_{pq} \\
	0 & A_{qp} & A_{qq} \end{bmatrix} 
	\begin{bmatrix} u_t \\ p_t \\ q_t \end{bmatrix} =
		\begin{bmatrix} b_u \\ b_p \\ b_q \end{bmatrix}
			\label{MFE}
\end{equation}
where $\Delta t$ is the time discretization step and:
\begin{subequations}
\begin{align}
  A_{uu} &= (\nabla^s \tensorTwo{\eta}, \tensorFour{C}_{\text{dr}} : \nabla^s \tensorTwo{\eta})_\Omega, &
	A_{up} &= -(\text{div} \; \tensorTwo{\eta}, b \chi)_\Omega, & \\
  A_{pu} &= (\chi, b \text{div} \; \tensorTwo{\eta})_\Omega, &
	A_{pp} &= (\chi, S_{\epsilon} \chi)_\Omega, &
  A_{pq} &= (\chi, \text{div} \; \tensorTwo{\phi})_\Omega, \\
  &&
	A_{qp} &= -(\text{div} \; \tensorTwo{\phi}, \chi)_\Omega, &
  A_{qq} &= ( \tensorTwo{\phi}, \mu \tensorTwo{\kappa}^{-1} \cdot \tensorTwo{\phi})_\Omega,
\end{align}
\label{eq:MFE_matrix}\null
\end{subequations}
with:
\begin{eqnarray}
	b_u &=& ( \tensorTwo{\eta}, \bar{\tensorOne{t}}_t )_{\Gamma_\sigma}, \label{bu} \\
	b_p &=& (\chi,b \text{div}\;u_{t-1} + S_{\epsilon} p_{t-1} + \Delta t f_t)_{\Omega}, \label{bp} \\
	b_q &=& -(\tensorTwo{\phi} \cdot \tensorOne{n}, \bar{p}_t)_{\Gamma_p}. \label{bq}
\end{eqnarray}
Notice that $A_{uu}$ and $A_{qq}$ are SPD, $A_{pp}$ is SPSD since $S_{\epsilon}$ can be 0 in the limiting case of
incompressible fluid, $A_{up}=-A_{pu}^T$, and $A_{pq}=-A_{qp}^T$. It is well-known (see, for instance~\cite{Rod_etal18})
that the $\mathbb{Q}_1-\mathbb{RT}_0-\mathbb{P}_0$ triplet can be unstable close to undrained conditions, i.e.,
small permeability or time integration step with incompressible fluid. To avoid such an occurrence, a stabilization
contribution can be introduced to annihilate the spurious checkerboard modes arising in the pressure 
solution. For example, here we modify the $A_{pp}$ term by adding an SPSD jump-jump stabilization matrix $A_{stab}$ computed
over sets of macro-elements, as proposed in \cite{Ferronato2021419,Frigo202136}.

It is easy to verify that the fully-discrete problem \eqref{MFE} can be just slightly modified to possess the double 
saddle-point structure \eqref{Eq1}. In fact, by multiplying by $-1$ and $\Delta t$ the second and third row, respectively,
we build the system matrix in \eqref{Eq1} with:
\begin{equation}
	A=A_{uu}, \quad B=-A_{pu}, \quad C=\Delta t A_{pq}, \quad D=A_{pp}+A_{stab}, \quad E=\Delta t A_{qq}.
	\label{MFE_dsp}
\end{equation}

\subsection{Mixed Hybrid Finite Element (MHFE) discretization} 

The hybrid formulation is obtained by replacing the discrete function space $\boldsymbol{\mathcal{Q}}^h$ for Darcy's
velocity representation with that obtained by using discontinuous piecewise polynomials defined on every element
$T$ of the partition $\mathcal{T}_h$:
\begin{equation}
\boldsymbol{\mathcal{W}}^h = \{\tensorOne{w}^h \in [L^{2}(\Omega)]^d \ \ | \ \
   \tensorOne{w}^h|_T \in {[\mathbb{RT}_0(T)]}, \; \forall T \in \mathcal{T}_h  \}.
	\label{eq:spaceW}
\end{equation}
The continuity of normal fluxes through inter-element edges or faces is enforced by Lagrange multipliers $\pi$, which are 
discretized through piecewise constant functions over each edge or face $e$. From a physical viewpoint, $\pi$
represent a pressure value, hence for consistency, it must be $\pi|_{\Gamma_p}=\bar{p}$. Denoting by $\mathcal{B}^h$ 
the finite-dimensional function space for $\pi$ and by $\mathcal{B}_0^h$ its counterpart taking homogeneous values
along $\Gamma_p$, the semi-discrete weak form \eqref{eq:IBVP_cW} becomes: find $\{\tensorOne{u}(t), \tensorOne{w}(t),
p(t), \pi(t)\}\in\boldsymbol{\mathcal{U}}^h\times\boldsymbol{\mathcal{W}}^h\times\mathcal{P}^h\times\mathcal{B}^h$
such that $\forall t\in]0,t_{\max}[$:
\begin{subequations}
\begin{align}
  &{(\nabla^s \tensorTwo{\eta}, \tensorFour{C}_{\text{dr}}:\nabla^s \tensorOne{u})}_{\Omega}-{(\text{div}\;\tensorTwo{\eta},bp)}_{\Omega}
   ={(\tensorTwo{\eta}, \bar{\tensorOne{t}})}_{\Gamma_\sigma} &&\forall \tensorTwo{\eta} \in \boldsymbol{\mathcal{U}}^h_0,\\
  &{(\tensorTwo{\varphi}, \mu \tensorTwo{\kappa}^{-1} \cdot \tensorOne{w})}_{\Omega} - \sum_{T\in\mathcal{T}_h}\left [{(\text{div} \;
   \tensorTwo{\varphi}, p)}_{T}
   - {(\tensorTwo{\varphi} \cdot \tensorOne{n}_e, \pi)}_{\partial T}\right ] = 0
   &&\forall \tensorTwo{\varphi} \in \boldsymbol{\mathcal{W}}^h, \\
	&{(\chi, b \; \text{div} \; \dot{{\tensorOne{u}}} )}_{\Omega} + \sum_{T\in\mathcal{T}_h} {(\chi, \text{div} \; \tensorOne{w})}_{T} + {(\chi, S_{\epsilon} \dot{p})}_{\Omega} =
   {(\chi,{f})}_{\Omega} &&\forall \chi \in \mathcal{P}^h, \\
  &\sum_{T \in \mathcal{T}_h} {(\zeta^h, \tensorOne{w} \cdot \tensorOne{n}_e)}_{\partial T}={(\zeta,\bar{q})}_{\Gamma_q} && \forall \zeta \in \mathcal{B}_0^h.
\end{align}\label{eq:IBVP_WH}\null
\end{subequations}
In \eqref{eq:IBVP_WH}, $\{\tensorTwo{\varphi}\}$ and $\{\zeta\}$ are the set of basis functions of the spaces 
$\boldsymbol{\mathcal{W}}^h$ and $\mathcal{B}^h$, respectively. The MHFE fully discrete weak form 
is finally obtained by introducing a backward Euler discretization of the time variables. Denoting by $w_t\in\mathbb{R}^{2p}$
and $\pi_t\in\mathbb{R}^p$ the vectors of discrete velocity and Lagrange multiplier values at time $t$, from \eqref{eq:IBVP_WH}
we obtain:
\begin{equation}
	\begin{bmatrix}
		A_{uu} & 0 & A_{up} & 0 \\
		0 & A_{ww} & A_{wp} & A_{w\pi} \\
		A_{pu} & \Delta t A_{pw} & A_{pp} & 0 \\
	0 & A_{\pi w} & 0 & 0 \end{bmatrix}
	\begin{bmatrix} u_t \\ w_t \\ p_t \\ \pi_t \end{bmatrix} =
		\begin{bmatrix} b_u \\ 0 \\ b_p \\ b_\pi \end{bmatrix},
			\label{MHFE_full}
\end{equation}
where the additional matrices with respect to \eqref{eq:MFE_matrix} are:
\begin{subequations}
\begin{align}
  A_{ww} &= ( \tensorTwo{\varphi}, \mu \tensorTwo{\kappa}^{-1} \cdot \tensorTwo{\varphi})_\Omega , &
        A_{wp} &= -\sum_{T\in\mathcal{T}_h}(\text{div} \; \tensorTwo{\varphi}, \chi)_T, & 
	A_{w\pi} &= \sum_{T\in\mathcal{T}_h}( \tensorTwo{\varphi} \cdot \tensorOne{n}_e, \zeta )_{\partial T} \\
	A_{pw} &= \sum_{T\in\mathcal{T}_h} (\chi, b \text{div} \; \tensorTwo{\varphi})_T, && \\
	A_{\pi w} &= \sum_{T\in\mathcal{T}_h} (\zeta, \tensorTwo{\varphi} \cdot \tensorOne{n}_e)_{\partial T}, &&
\end{align}
\label{eq:MHFE_matrix}\null
\end{subequations}
and:
\begin{equation}
	b_\pi = {(\zeta,\bar{q}_t)}_{\Gamma_q}.
	\label{bpi}
\end{equation}
Notice that $A_{ww}$ is SPD and block diagonal, while $A_{wp}=-A_{pw}^T$ and $A_{w\pi}=A_{\pi w}^T$. With the low-order
triplet $\mathbb{Q}_1-\mathbb{RT}_0-\mathbb{P}_0$ a stabilization is still necessary, so the SPSD contribution $A_{stab}$
is added to $A_{pp}$ as with the MFE discretization. The block system \eqref{eq:MHFE_matrix} can be reduced by static 
condensation eliminating the $w_t$ variables:
\begin{equation}
	\begin{bmatrix}
		A_{uu} & A_{up} & 0 \\
		A_{pu} & A_{pp} + A_{stab} - \Delta t A_{pw} A_{ww}^{-1} A_{wp} & \Delta t A_{pw} A_{ww}^{-1} A_{w\pi} \\
	0 & A_{\pi w} A_{ww}^{-1} A_{wp} & A_{\pi w} A_{ww}^{-1} A_{w\pi} \end{bmatrix}
	\begin{bmatrix} u_t \\ p_t \\ \pi_t \end{bmatrix} =
		\begin{bmatrix} b_u \\ b_p \\ -b_\pi \end{bmatrix}.
			\label{MHFE}
\end{equation}
As with the MFE discretization, the fully-discrete problem \eqref{MHFE} can be easily modified to possess the double 
saddle-point structure \eqref{Eq1}. By multiplying by $-1$ and $\Delta t$ the second and third row, respectively,
we build the matrix in \eqref{Eq1} with:
\begin{equation}
        A=A_{uu}, \; B=-A_{pu}, \; C=\Delta t A_{\pi w} A_{ww}^{-1} A_{wp}, \; D=A_{pp}+A_{stab}-
	\Delta t A_{pw} A_{ww}^{-1} A_{pw}^T, \; E=\Delta t A_{\pi w} A_{ww}^{-1} A_{w\pi}.
        \label{MHFE_dsp}
\end{equation}

\section{Numerical results}
A set of numerical experiments, concerning the mixed form of Biot's poroelasticity equations introduced in Section \ref{SecBiot}, is used to investigate the quality of the bounds on the eigenvalues of the preconditioned matrix and the effectiveness of the proposed triangular preconditioners on a problem with increasing size of the matrices involved. 
As a reference test case for the experiments, we considered the porous cantilever beam problem originally introduced in \cite{WheelPhil} and already used in \cite{Frigo202136}. The domain is the unit square or cube for the 2-D and 3-D case, respectively, with no-flow boundary conditions along all sides, zero displacements along the left edge, and a uniform load applied on top. Material properties are summarized in Table \ref{prop}.

\begin{table}
	\caption{Cantilever beam test case: material properties.}
	\label{prop}
\begin{center}
\begin{tabular}{lll}
	Parameter  & Value  & Unit \\
	\hline
    Time integration step ($\Delta t$)& $1\times 10^{-5}$ & [s] \\
	Young’s modulus ($E$)& $\red 1 \times 10^{5}$ & [Pa] \\
	Poisson’s ratio ($\nu$)& 0.4 & [--] \\
	Biot’s coefficient ($b$)& 1.0 & [--] \\
	Constrained specific storage ($S_\varepsilon$)& 0 & [Pa] \\
	Isotropic permeability ($\kappa$) &$1 \times 10^{-7}$ & [$m^2$] \\
	Fluid viscosity ($\mu$)& $1 \times 10^3$ & [Pa $\cdot$ s] \\
\end{tabular}
\end{center}
\end{table}

First, the 2-D version of the cantilever beam problem is used to compute the full eigenspectrum of the preconditioned matrix $\mathcal{A}\mathcal{P}^{-1}$ with both the MFE and MHFE discretizations of the mixed Biot's poroelasticity equations. The eigenspectrum is compared against the bounds introduced in Section \ref{sec:eig_tri} for a relatively small-sized matrix, thus validating the theoretical outcomes. The bounds on the eigenvalues of the preconditioned matrix $\mathcal{A}\mathcal{P}_D^{-1}$ given in Section \ref{sec:eig_diag} are also used to compare the expected performance of GMRES and MINRES. Then, the 3-D version of the cantilever beam problem, discretized by the MHFE formulation, is used with increasing size of the test matrices to investigate the efficiency and scalability of the proposed preconditioners.

\subsection{Eigenvalue bounds validation}
Let us consider a uniform discretization of the unit square with spacing $h=1/40$, such that
the sizes of the matrix blocks in $\mathcal{A}$ are $n=3362$, $m=1600$, and $p=3200$.
To check the relative error, 
we use a randomly manufactured solution with the corresponding right-hand side.
Both a right-preconditioned GMRES with $\mathcal{P}$ and a preconditioned MINRES with $\mathcal{P}_D$ are used as a solver.
The exit test is on the 2-norm of the residual vector normalized to its initial value, with a tolerance equal to either $10^{-13}$ (right-preconditioned GMRES) or $10^{-10}$ (preconditioned MINRES),
to obtain comparable relative errors in the solution. The initial guess is the null vector.
{
For the numerical experiments, we use a MATLAB implementation run on an Intel Core Ultra 7 165U Notebook at 3.8 GHz with 16 GB RAM.}

\subsubsection{MFE discretization}
The system matrix \eqref{Eq1} 
has the blocks defined as in \eqref{MFE_dsp}.
To set up the preconditioner, we
have to define the approximations $\widehat A$, $\widehat S$, and $\widehat X$. 

Since $A$ is a standard elasticity matrix, an appropriate off-the-shelf approximation $\widehat A$ can be computed by using an Algebraic MultiGrid (AMG) tool. In this case, we used the Matlab function \texttt{hsl\_mi20}.
Regarding the approximation $\widehat S$ of the first-level Schur complement matrix $\widetilde S=D+B\widehat{A}^{-1}B^T$, we first observe that it is not possible to compute explicitly $\widetilde S$ because $\widehat{A}^{-1}$ is not available. 
A simple and common purely algebraic strategy with saddle-point problems relies on replacing $\widetilde S$ with:
\begin{equation}
    \widetilde S^{(1)} = D + B (\text{diag} (A ))^{-1} B^T,
    \label{eq:tS1}
\end{equation}
then setting $\widehat S$ as the incomplete Cholesky factorization of $\widetilde S^{(1)}$ with no fill-in.
For the second-level Schur complement matrix $\widetilde X=E+C\widehat{S}^{-1}C^T$, we can follow a similar strategy, where $\widehat{S}^{-1}$ is approximated with the inverse of $\text{diag}(\widetilde{S}^{(1)})$ and then $\widehat X$ is the incomplete factorization of $\widetilde X$ with no fill-in.
We denote the block triangular preconditioner built with these choices as $\mathcal{P}^{(1)}$, while the corresponding block diagonal option is $\mathcal{P}^{(1)}_D$.

In this case, the intervals $I_*$ for the indicators $\gamma_*$ read:

\vspace{0.5cm}
\begin{center}
	\begin{tabular}{lllllllll}
		 $I_A=$ & $[5.01\times 10^{-5}$, & $ 1.035]$  \\[.4em]
		 $I_S=$ & $[0.448$, & $ 3.515]$  & \
		 $I_R=$ & $[7.7\times 10^{-3} $, & $ 3.470]$  & \
		 $I_D=$ & $[5.3\times 10^{-3} $, & $ 0.983]$ \\[.4em]
		 $I_X=$ & $[0.999             $, & $ 1.030]$ & \
		 $I_K=$ & $[4.3\times 10^{-5} $, & $ 0.035]$  & \
		 $I_E=$ & $[0.995            $  & $ 1.000]$ 
	\end{tabular}
	\end{center}

\vspace{0.5cm}
\noindent
Using Theorem \ref{Theorem_real} and \ref{Theorem_real_D},
the indicators above give the eigenvalue bounds for the real eigenvalues reported in Table \ref{tab:eigen}. These bounds are compared to the exact spectral intervals of $\mathcal{A}\mathcal{P}^{(1),-1}$ and $\mathcal{A}\mathcal{P}_D^{(1),-1}$.
We observe that the accuracy of the bounds is extremely good; they are sufficiently tight, particularly for the eigenvalues that are close to 0.

\begin{table}
        \label{tab:eigen} 
        \begin{center}
        \caption{Eigenvalues and bounds of the preconditioned matrix for the block triangular and block diagonal preconditioners.}
\begin{tabular}{lcc}
    & Bounds & True eigenvalues \\
        \hline
        $\mathcal{A}\mathcal{P}^{(1),-1}$& [5.0098e-05, 5.5536] & [5.0113e-05, 3.4649] \\
        $\mathcal{A}\mathcal{P}_D^{(1),-1}$& [-2.8548, \textcolor{red}{-0.3021}]$\cup$[5.0098e-05, 3.9014] & [-1.4950, -0.3652]$\cup$[5.0105e-05, 2.3850] \\
    \hline
    $\mathcal{A}\mathcal{P}^{(1,\omega),-1}$& [5.0098e-05, 2.3865] & [5.0113e-05, 1.0346] \\
        $\mathcal{A}\mathcal{P}_D^{(1,\omega),-1}$ & [-0.9672, \textcolor{red}{-0.0415}]$\cup$[5.0098e-05, 2.6262] & [-0.2971,  -0.0440]$\cup$[5.0105e-05, 1.2381] \\
        \hline 
    $\mathcal{A}\mathcal{P}^{(2),-1}$ & [5.0098e-05, 3.5249] & [5.0113e-05, 1.4321] \\ 
        $\mathcal{A}\mathcal{P}_D^{(2),-1}$ & [-1.7058, \textcolor{red}{-0.2978}]$\cup$[5.0098e-05, 3.1750] & [-0.9295, -0.3608]$\cup$[5.0105e-05, 1.6997] \\        \hline
    $\mathcal{A}\mathcal{P}^{-1}$ & \r [5.0098e-05, 3.3065] & \r [5.0113e-05, 1.0346] \\
        $\mathcal{A}\mathcal{P}_D^{-1}$ & \r [-1.3486, \textcolor{red}{-0.2166}]$\cup$\r [5.0098e-05, 3.0332] & \r [-0.8323, -0.2826]$\cup$\r [5.0105e-05, 1.415] \\
\end{tabular}
\end{center}
\end{table}

The real spectral interval turns out to be quite wide for both preconditioners, especially in consideration of the small size of the problem at hand.
The reason for this occurrence is twofold: (i) the interval $I_A$ has a very small left extreme, and (ii) the intervals $I_A$ and $I_S$ do not overlap. 
While the first issue can be simply addressed by improving the quality of the off-the-shelf AMG approximation of $A$, the second issue can be fixed by scaling $\widehat S$ by a factor $\omega^{-1} > 1$, similarly to the strategy used in \cite{bm11cmame}. 
This will proportionally decrease $\gamma_{\max}^S$ and $\gamma_{\max}^R$, thus reducing the real spectral
interval for the triangular preconditioner. Hence, we define
$ \widehat S(\omega) = \omega^{-1} \widehat S,$ and use, for instance, $\omega = 10^{-1}$, with the new preconditioners denoted as $\mathcal{P}^{(1,\omega)}$ and $\mathcal{P}_D^{(1,\omega)}$.
The new indicator intervals become:

\vspace{0.5cm}
\begin{center}
	\begin{tabular}{lllllllll}
		 $I_S=$ & $[0.045$, & $ 0.352]$  & \
		 $I_R=$ & $[7.7\times 10^{-4} $, & $ 0.347]$  & \
		 $I_D=$ & $[5.3\times 10^{-4} $, & $ 0.098]$ \\[.4em]
		 $I_X=$ & $[0.999             $, & $ 1.030]$ & \
		 $I_K=$ & $[4.3\times 10^{-6} $, & $ 0.004]$  & \
		 $I_E=$ & $[0.999            $,  & $ 1.000]$ 
	\end{tabular}
\end{center}

\vspace{0.5cm}
\noindent
from which we obtain the bounds for the real eigenvalues shown in Table \ref{tab:eigen}.
It can be seen that the spectral interval for the real positive eigenvalues has been narrowed.
In contrast, reducing the indicators $\gamma_*^S, \gamma_*^R$ and $\gamma_*^D$ slightly widens the region of purely complex eigenvalues, as shown in Figure \ref{complex}. 
In total, the behavior of the Krylov solvers is positively influenced by the reduced real spectral intervals, as confirmed by the convergence results shown in Table \ref{tab:results}.

\begin{figure}
		\label{complex}
	\includegraphics[width=8cm]{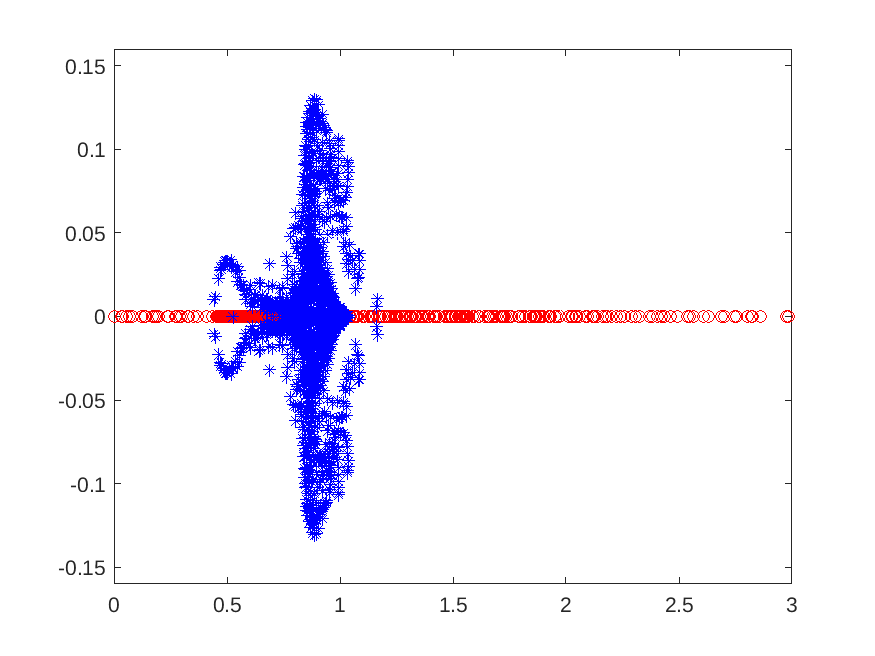}
		\hspace{-4mm}
	\includegraphics[width=8cm]{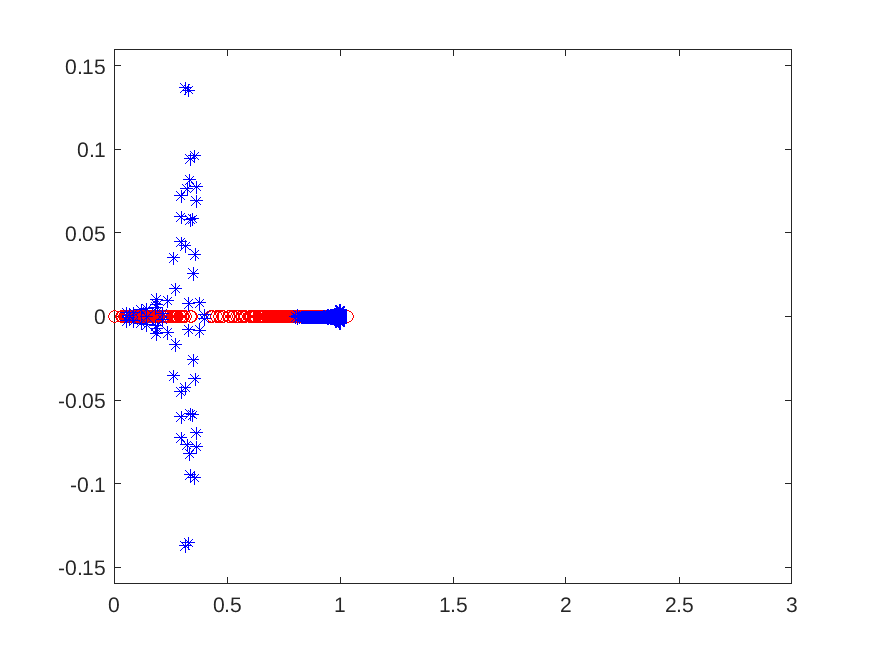}
		\caption{MFE discretization: eigenvalue distribution for $\mathcal{A}\mathcal{P}^{(1),-1}$ (left) and $\mathcal{A}\mathcal{P}^{(1,\omega),-1}$ with $\omega = 10^{-1}$ (right). The red and blue markers denote the real and complex eigenvalues, respectively.}
\end{figure}

\begin{table}
    \label{tab:results}
    \caption{Number of iterations $n_{it}$ and CPU time $T$ of right-preconditioned GMRES and preconditioned MINRES with block triangular and block diagonal preconditioners.}
    \begin{center}
    \begin{tabular}{l|lrr|lrr}
    & \multicolumn{3}{c|}{GMRES} & \multicolumn{3}{c}{MINRES} \\
    & prec. & $n_{it}$ & $T$ [s] & prec. & $n_{it}$ & $T$ [s] \\
    \hline    
	& $\mathcal{P}^{(1)}$ & 85 & 0.95 & $\mathcal{P}_D^{(1)}$ & 227 & 2.12 \\
	  MFE & $\mathcal{P}^{(1,\omega)}$ & 62 & 0.69 & $\mathcal{P}_D^{(1,\omega)}$ & 167 & 1.60 \\
    & $\mathcal{P}^{(2)}$ & 62 & 0.72 & $\mathcal{P}_D^{(2)}$ & 142 & 1.39 \\
    \hline
    MHFE & $\mathcal P$ & \r 54 &  \r 0.72 & $\mathcal{P}_D$ & \r 148 & \r 1.51 \\
    \end{tabular}
    \end{center}
\end{table}
In alternative to equation \eqref{eq:tS1}, a second way to approximate the first-level Schur complement relies on replacing $\widetilde S$ with: 
\begin{equation}
    \widetilde{S}^{(2)} = D + S_K,
    \label{eq:tS2}
\end{equation}
where $S_K$ is the diagonal ``fixed-stress'' matrix, introduced as a preconditioner in \cite{WCT2016,CasWhiFer16}, which is a suitable scaling of the lumped mass matrix of the pressure shape functions. The lumped structure of the approximation follows from the nature of the differential operators underlying $A$ and $B$, which represent a discrete div-grad and a discrete divergence operator, respectively, while the value of the diagonal entries of $S_K$:
\begin{equation}
    [S_K]_i = \frac{b^2 \left|\Omega_i\right|}{\overline{K}_i}, \qquad i=1,\ldots,p,
    \label{eq:SKi}
\end{equation}
where $|\Omega_i|$ is the measure of the $i$-th finite element and $\overline{K}_i$ is an estimate of the associated bulk modulus, is dictated by physical considerations \cite{WCT2016}.
A more general algebraic strategy for computing the entries \eqref{eq:SKi} can be implemented following the ideas sketched in \cite{CasWhiFer16}. Let us denote with $\mathcal{N}^{(i)}$ the subset of $\mathcal{N}=\{1,2,\ldots,n\}\subset\mathbb{N}$ containing the indices of the non-zero components of the $i$-th row $B^{(i)}$ of $B$, and define $R^{(i)}$ as the restriction operator from $\mathbb{R}^n$ to $\mathbb{R}^{|\mathcal{N}^{(i)}|}$ such that:
\begin{equation}
    A^{(i,i)} = R^{(i),T} A R^{(i)}
    \label{eq:Aii}
\end{equation}
is the sub-matrix of $A$ made by the entries lying in the rows and columns with indices in $\mathcal{N}^{(i)}$. Since $A^{(i,i)}$ is a diagonal symmetric block of an SPD matrix, it is non-singular and can be regularly inverted. The $i$-th diagonal entry of $S_K$ can be algebraically computed as:
\begin{equation}
    [S_K]_i = r(B^{(i)}) A^{(i,i),-1} r(B^{(i),T}), \qquad i=1,\ldots,p,
    \label{eq:algSKi}
\end{equation}
where $r(\cdot)$ is the operator restricting a vector to its non-zero entries.

Since $S_K$ is diagonal, the first-level Schur complement approximation $\widetilde{S}^{(2)}$ of equation \eqref{eq:tS2} is diagonally dominant and $\widehat S$ can be simply set as $\text{diag}(\widetilde{S}^{(2)})$. As a consequence, the second-level Schur complement matrix $\widetilde{X}=E+C\widehat{S}^{-1}C^T$ can be explicitly computed, with $\widehat{X}$ defined again as the incomplete factorization of $\widetilde{X}$ with no fill-in. The block triangular and block diagonal preconditioners built with these choices as denoted as $\mathcal{P}^{(2)}$ and $\mathcal{P}_D^{(2)}$, respectively.

Using $\widetilde S^{(2)}$ yields the following intervals for the indicators $\gamma_*$:

\vspace{0.5cm}
\begin{center}
	\begin{tabular}{lllllllll}
		 $I_A=$ & $[5.01\times 10^{-5}$, & $ 1.035]$  \\[.2em]
		 $I_S=$ & $[0.439$, & $ 1.489]$  & \
		 $I_R=$ & $[3.4\times 10^{-3} $, & $ 1.310]$  & \
		 $I_D=$ & $[4.7\times 10^{-3} $, & $ 0.559]$ \\[.2em]
		 $I_X=$ & $[0.998             $, & $ 1.001]$ & \
		 $I_K=$ & $[4.0\times 10^{-6} $, & $ 0.005]$  & \
		 $I_E=$ & $[0.993            $  & $ 1.003]$
	\end{tabular}
\end{center}
	
\vspace{0.5cm}
\noindent
The bounds and the eigenvalues of the preconditioned matrices $\mathcal{A}\mathcal{P}^{(2)}$ and $\mathcal{A}\mathcal{P}_D^{(2)}$ are provided in Table \ref{tab:eigen}. 
By comparing the bounds with those for $\mathcal{A}\mathcal{P}^{(1,\omega),-1}$ we can predict a similar GMRES convergence,   while the more clustered negative eigenvalues suggest an improved MINRES performance with respect to $\mathcal{P}_D^{(1,\omega)}$.
Both observations are confirmed in
Table \ref{tab:results}. 
Hence, we can conclude that the use of the fixed-stress approximation for the first-level Schur complement is generally helpful. 

\subsubsection{MHFE discretization}
\label{subsection:MHFEresults}
The system matrix \eqref{Eq1} now has the block definition \eqref{MHFE_dsp}.
To set up the preconditioner, we use as $\widehat A$ the same \texttt{hsl\_mi20} AMG operator as before,
while, following the outcome of the MFE numerical experiments, we use the diagonal of the ``fixed-stress'' approximation \eqref{eq:tS2} for $\widehat S$ and the incomplete Cholesky factorization of $\widetilde X$ with no fill in for $\widehat X$. 
The block triangular and block diagonal preconditioners built with these choices are simply denoted as $\mathcal{P}$ and $\mathcal{P}_D$, respectively.

\begin{figure}
	\label{eigpic}
	\centerline{\includegraphics[width=9cm]{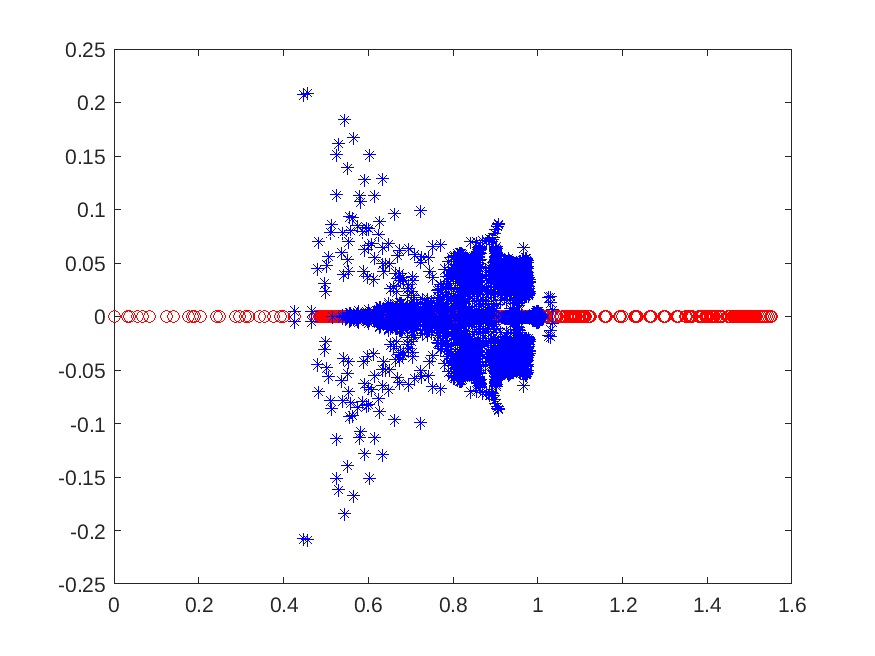}}
	\caption{MHFE discretization: eigenvalue distribution for $\mathcal{A}\mathcal{P}^{-1}$. The red and blue markers denote the real and complex eigenvalues, respectively.}
\end{figure}

The intervals $I_*$ for the indicators $\gamma_*$ read:

\vspace{0.5cm}
\begin{center}
	\begin{tabular}{lllllllll}
		 $I_A=$ & $[5.01\times 10^{-5}$, & $ 1.035]$  \\[.4em]
		 $I_S=$ & $[0.286            $, & $ 1.086]$  & \
		 $I_R=$ & $[3.4\times 10^{-3} $, & $ 0.655]$  & \
		 $I_D=$ & $[3.1\times 10^{-3} $, & $ 0.536]$ \\[.4em]
		 $I_X=$ & $[\r 0.797            $, & $ \r 1.186]$ & \
		 $I_K=$ & $[\r 0                 $  & $ \r 0.005]$  & \
		 $I_E=$ & $[\r 0.795            $  & $ \r 1.186]$
	\end{tabular}
\end{center}
\vspace{0.5cm}
\noindent
\textcolor{mygreen}{from which Theorem \ref{Th6} and Remark 1 yield the bound $|\lambda - 1| \le    \sqrt{1-\gamma_{\min}^D} = 0.998$ for the complex eigenvalues.}
As for the real eigenvalues, we have the bounds and true spectral intervals reported in Table \ref{tab:eigen}.
The eigenvalues of $\mathcal{A}\mathcal{P}^{-1}$ are shown in Figure \ref{eigpic}.
The GMRES and MINRES convergence, shown in Table \ref{tab:results}, is consistent with the eigenvalue distribution, showing an overall similar algebraic behavior between the MFE and MHFE discretizations.

\subsection{Computational efficiency and scalability}
We consider a set of uniform hexahedral MHFE discretizations of the unit cube with spacing $h=0.100, 0.050, 0.025$, which gives rise to the double saddle-point linear systems \eqref{Eq1} whose size and number of nonzeros are reported in Table \ref{tab:size}. 
The linear systems are solved by using either GMRES or MINRES preconditioned with $\mathcal{P}$ or $\mathcal{P}_D$, respectively.
{
	We considered a right-hand side corresponding to an imposed solution of all ones, and the zero vector as the initial guess.}

\begin{table}
        \caption{3D cantilever beam problem on a unit cube: size and number of nonzeros of the test matrices ($N=n+m+p$). }
        \label{tab:size}
\begin{center}
\begin{tabular}{r|rrrrr}
        $1/h$&$n$  & $m$  &$p$   &$N$  & nonzeros\\
        \hline            
        10 & 3993  &\red 1000  &\red 3300  &8293   & 0.3$\times 10^6$\\
        20 &27783  &\red 8000  &\red 25200 &60983  &2.4$\times 10^6$  \\
        40 &206763 &\red 64000 &\red 196800&467563 & 19.2$\times 10^6$ \\
\end{tabular}
\end{center}
\end{table}

\begin{figure}
	\begin{center}
		\includegraphics[width=8cm]{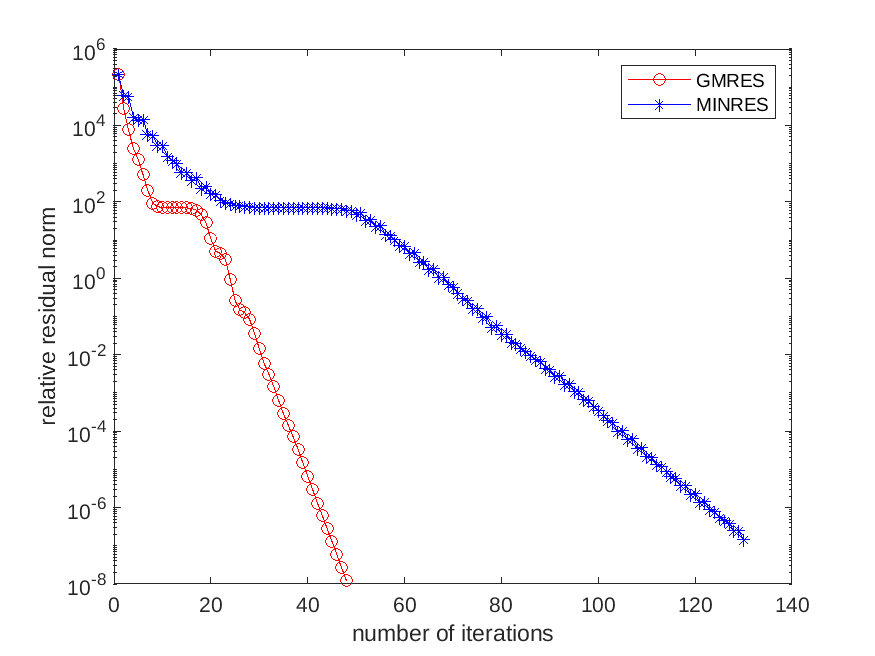}
	\end{center}
		\caption{3D cantilever beam problem on a unit cube: convergence profile of MINRES and GMRES for the finest mesh.}
		\label{convprof40}
\end{figure}

In addition to the \texttt{hsl\_mi20} routine previously mentioned, for $\widehat A$ we also use the Matlab implementation of the
AMG operator belonging to the Chronos software package \cite{Janna1, Janna2}, which is more suited for approximating the elasticity matrix.
Chronos provides a classical AMG implementation with a smoother based on
an adaptive variant of the factorized sparse approximate inverse preconditioner \cite{Jan_etal15}. For the elasticity matrix, the prolongation is
constructed through a least square fit of the rigid body modes and
further improved via energy minimization \cite{Janna3}. The approximations $\widehat S$ and $\widehat X$ of the Schur complements are as in Subsection \ref{subsection:MHFEresults}.

\begin{table}
	\caption{3D cantilever beam problem on a unit cube: number of iterations $n_{it}$ and CPU time $T$ with different AMG approximations for $\widehat A$ obtained from either the HSL routine \texttt{hsl\_mi20} or the Chronos software package. 
	}
	\label{scalability}
\begin{center}
\begin{tabular} {l|c|rrr|rrr}
	&& \multicolumn{3}{c|}{HSL} & \multicolumn{3}{c}{Chronos} \\
	$1/h$ & solver & $n_{it}$ & $T$ [s] & rel. err. & $n_{it}$ & $T$ [s] & rel. err.  \\
\hline
	$10$ & GMRES &  27 & 0.15 & $1.7 \times 10^{-5}$ & 46  & 0.10 & $1.58 \times 10^{-6}$\\
	       & MINRES&  79 & 1.06 & $2.4 \times 10^{-5}$ & 111 & 0.08 & $1.71 \times 10^{-6}$\\
       \hline
	$20$ & GMRES &  38 & 2.63 &$1.0 \times 10^{-5}$& 47  & 0.61 & $1.60 \times 10^{-6}$ \\
	       & MINRES& 135 & 7.86&$1.1 \times 10^{-5}$& 122 & 1.05 & $1.44 \times 10^{-6}$ \\
       \hline
	$40$ & GMRES & 71  &37.26&$1.0 \times 10^{-5}$& 47  & 4.48 & $1.17 \times 10^{-6}$ \\
	       & MINRES& 264 &127.55 & $5.5 \times 10^{-6}$& 129 & 9.21 & $1.09 \times 10^{-6}$ \\
\end{tabular}
\medskip
\end{center}
\end{table}

The convergence results 
are summarized in Table \ref{scalability}. 
We observe that the AMG preconditioner from the Chronos software package performs better than the HSL package in terms of scalability, which is almost perfect with both GMRES and MINRES. This outcome also shows that the approximations introduced for the first- and second-level Schur complements do not affect the overall scalability of $\mathcal{P}$ and $\mathcal{P}_D$, which depends only on the quality of $\widehat A$ as a preconditioner for $A$.
The Preconditioned Conjugate Gradient (PCG) method for solving the linear system with the $(1,1)$ block $A$, accelerated by the Chronos AMG operator, requires 25, 26, and 27 iterations, for $h=0.100, 0.050$ and $0.025$, respectively. 
Therefore, no loss of scalability is introduced by the (simple) approximations of the two Schur complements proposed in this work.
Finally, the results confirm the superior computational efficiency of GMRES over MINRES for our setting, in terms of both convergence rate and CPU time (see also Figure \ref{convprof40}, {where the absolute residual is plotted vs the iteration number}) to obtain
roughly the same relative error, {see the last column of Table \ref{scalability}}.

\section{Conclusions}
Double-saddle point linear systems naturally arise in the monolithic solution to the coupled Biot's poroelasticity equations written in mixed form. 
Concerning the general $3\times 3$ block structure of the matrix $\mathcal{A}$ in equation \eqref{Eq1}, the block $A$ is an SPD elasticity matrix, $D$ and $E$ are mass matrices where $D$ can be singular, $B$ is a discrete divergence, and $C$ a scaled Gram matrix. A family of block triangular preconditioners $\mathcal{P}$ has been developed for this double saddle-point problem by combining physically-based and purely algebraic Schur complement approximations.

In this work, we have analyzed the spectral properties of the preconditioned matrix $\mathcal{A}\mathcal{P}^{-1}$ by extending the results already available for the simpler case with $D=E=0$. 
We have proved that:
\begin{itemize}
    \item the complex eigenvalues, if any, lie in a circle with center $(1,0)$ and radius smaller than 1;
    \item the real eigenvalues are the roots of a third-degree polynomial with real coefficients and are bounded by quantities that depend on the quality of the inner preconditioners $\widehat A$, $\widehat S$, and $\widehat X$ of the $(1,1)$ block $A$, and the first- and second-level Schur complements $S$ and $X$;
    \item the eigenvalues obtained by using a block diagonal preconditioner are clustered in two intervals, one in the negative and one in the positive side of the real axis, with the bounds of the positive interval similar to those obtained with the block triangular preconditioner.
\end{itemize}
The numerical experiments on a set of test problems arising from mixed and mixed hybrid discretizations of Biot's poroelasticity equations have validated the established bounds, which generally turn out to be quite tight. The results that follow are worth summarizing.
\begin{enumerate}
    \item The use of GMRES with a block triangular preconditioner is usually more efficient than MINRES with a block diagonal preconditioner, both in terms of convergence rate and computation time.
    \item Complex eigenvalues have limited influence on the solver convergence. The key factor for the overall performance is the approximation of the elasticity matrix, which controls the lower bound of the eigenspectrum of the preconditioned matrix. 
    \item Another factor that affects the solver convergence is the overlap between the spectral intervals of $\overline A$ and $\overline S$. This is automatically ensured by using a ``fixed-stress'' approximation of $B\widehat{A}^{-1}B^T$ in the computation of the first-level Schur complement. In addition, in this way, $\widehat S$ can be set as a simple diagonal matrix and $\widetilde X$ can be computed exactly. 
    \item The overall scalability of the proposed family of block triangular preconditioners for the mixed form of Biot's poroelasticity equations depends mainly on the scalability of $\widehat A$ as a preconditioner of the elasticity matrix $A$. In this context, the AMG operator provided by the Chronos software package appears to be a valuable tool.
\end{enumerate}

\subsection*{Acknowledgements}
The authors are members of the Gruppo Nazionale Calcolo Scientifico-Istituto Nazionale di Alta Matematica (GNCS-INdAM).
AM gratefully acknowledges the support of the GNCS-INdAM Project CUP$\_$E53C23001670001. The work of AM was carried out within the PNRR research activities of the consortium iNEST (Interconnected North-East Innovation Ecosystem) funded by the European Union Next-GenerationEU
(PNRR -- Missione 4 Componente 2, Investimento 1.5 -- D.D. 1058 23/06/2022, ECS$\_$00000043).
This manuscript reflects only the Authors' views and opinions; neither the European Union nor the European Commission can be considered responsible for them.

\end{document}